\documentclass[11pt]{amsart}
\usepackage{amsmath}
\usepackage{amssymb}
\usepackage{amsthm}
\usepackage{amscd}
\usepackage{amsfonts}
\usepackage{graphicx}
\usepackage{fancyhdr}
\usepackage[utf8]{inputenc}
\usepackage{mathrsfs}
\usepackage{tikz-cd}
\usepackage[all,cmtip]{xy}
\usepackage{hyperref}
\numberwithin{equation}{section}
\newtheorem{theorem}[equation]{Theorem}
\newtheorem{thm}{Theorem}
\theoremstyle{plain}
\newtheorem{lemma}[equation]{Lemma}
\newtheorem{proposition}[equation]{Proposition}
\newtheorem{definition}[equation]{Definition}
\newtheorem{corollary}[equation]{Corollary}

\newtheorem*{corollary*}{Corollary}

\newtheorem{remark}[equation]{Remark}

\def\Aut{\mathrm{Aut}}

\def\GL{\mathrm{GL}}
\def\GSp{\mathrm{GSp}}
\def\GSpin{\mathrm{GSpin}}

\def\SO{\mathrm{SO}}

\def\Nilp{\mathrm{Nilp}}
\def\ANilp{\mathrm{ANilp}}
\def\det{\mathrm{det}}
\def\ord{\mathrm{ord}}

\def\inv{\mathrm{inv}}

\def\red{\mathrm{red}}

\def\Sets{\mathrm{Sets}}
\def\Ker{\mathrm{Ker}}
\def\ssp{\mathrm{ssp}}
\def\ss{\mathrm{ss}}
\def\der{\mathrm{der}}
\def\max{\mathrm{max}}
\def\BT{\mathrm{BT}}

\def\KS{\mathrm{KS}}
\def\cris{\mathrm{cris}}
\def\rig{\mathrm{rig}}
\def\dR{\mathrm{dR}}

\DeclareMathOperator{\Hom}{Hom}

\DeclareMathOperator{\End}{End}
\DeclareMathOperator{\Adm}{Adm}

\def\calM{\mathcal{M}}
\def\calO{\mathcal{O}}

\def\CC{\mathbb{C}}

\def\FF{\mathbb{F}}

\def\QQ{\mathbb{Q}}
\def\RR{\mathbb{R}}

\def\ZZ{\mathbb{Z}}

\newcommand{\Mass}{\mathrm{Mass}}
\newcommand{\Vol}{\mathrm{Vol}}

\newcommand{\Spf}{\mathrm{Spf}}

\newcommand{\Dieu}{Dieudonn\'{e}}
\newcommand{\Tr}{\mathrm{Tr}}
\newcommand{\Nm}{\mathrm{Nm}}

\newcommand{\Fil}{\mathrm{Fil}}

\newcommand{\Gspin}{\mathrm{GSpin}}
\newcommand{\Isom}{\mathrm{Isom}}

\usepackage[top=1.4in, bottom=1.3in, left=1.3in, right=1.2in]{geometry}
\author{Haining Wang}
\address{\parbox{\linewidth} { Department of Mathematics,\\ McGill University,\\ 805 Sherbrooke St W,\\ Montreal, QC H3A 0B9, Canada.~ }}
\email{wanghaining1121@outlook.com}
\date{\today}

\begin{document}
\title{On the superspecial loci of orthogonal type Shimura varieties}


\begin{abstract}
In this note, we study the superspecial loci of orthogonal type Shimura varieties of signature $(n-2, 2)$ with $n\geq 3$. We prove a conjecture of Gross on the parametrizations of the superspecial locus in the special fiber of an orthogonal type Shimura variety and its lift in the integral model by certain homogeneous spaces.  As applications, we provide a mass formula for the superspecial locus. We also indicate how Gross's conjectures can be generalized to the Coxeter type Shimura varieties using the group theoretic method of G\"{o}rtz and He.
\end{abstract}

\maketitle

\tableofcontents
\section{Introduction}
Let $p$ be an odd prime. An abelian variety $A$ over a field of characteristic $p$ is said to be \emph{superspecial} if it is isomorphic to a product of supersingular elliptic curves under the base change to an algebraically closed field. Consider $\mathscr{A}_{g,\FF_{p}}$ the moduli space over $\FF_{p}$ of principal polarized abelian varieties. Let $\mathscr{A}_{g,\ssp}$ be the subset of $\mathscr{A}_{g,\FF_{p}}(\overline{\FF}_{p})$ which consists of superspecial points. This is a finite and closed subset of  $\mathscr{A}_{g}\otimes \overline{\FF}_{p}$. The structure of $\mathscr{A}_{g, \ssp}$ is well studied \cite{KO-COM87, Yu-doc06}. In particular there is a uniformization of  $\mathscr{A}_{g, \ssp}$ by the double coset $G^{\prime}(\QQ)\backslash G^{\prime}(\mathbb{A}_{f})/ G^{\prime}(\widehat{\ZZ})$ where $G^{\prime}$ is an inner form of the symplectic similitude group $\GSp_{2g}$. Using this uniformization, one can give a quantitive result about the size of superspecial locus known as the mass formula see \cite{Eke87, Yu-doc06}. This is a generalization of the famous Deuring's mass formula for supersingular elliptic curves which says
$$\sum_{[E]}\frac{1}{|\Aut(E)|}=\frac{p-1}{24}.$$ 
The study of superspecial locus and its mass formula has been generalized to PEL type Shimura varieties in \cite{Yu11}.  In this note we provide a generalization of the above theme to the context of orthogonal type Shimura varieties. Being an abelian type Shimura variety, an orthogonal type Shimura variety does not admit an immediate moduli interpretation and therefore it is not clear how to define the superspecial locus. In this note, we simply define the superspecial locus to be the \emph{minimal Ekedahl-Oort stratum} in the Ekedahl-Oort stratification of the orthogonal type Shimura variety. The Ekedahl-Oort stratifications for abelian type Shimura varieties are studied in \cite{SZ17} and its local analogues for Rapoport-Zink spaces are studied in \cite{Shen-gen}. We will not use their theory explicitly but rather our starting point is the \emph{Bruhat-Tits stratification} of the $\GSpin$-type Rapoport-Zink space investigated in \cite{HP17}. Our definition of the superspecial locus for the $\GSpin$-type Rapoport-Zink is the union of Bruhat-Tits strata of minimal type. The minimal Ekedahl-Oort stratum and the minimal Bruhat-Tits stratum in fact agrees by general theory explained in \cite{GH15}. Then the superspecial locus for the orthogonal type Rapoport-Zink space is simply the image of the  superspecial locus for the $\GSpin$-type Rapoport-Zink space under the natural \'{e}tale covering map. We show this agrees with the naive way of defining the superspecial locus which is simply restricting the superspecial locus of the Siegel type Rapoport-Zink space to the $\GSpin$-type, then project down to the orthogonal type Rapoport-Zink space. The corresponding superspecial locus for the orthogonal type Shimura variety can be defined in an analogues way and linked to the Rapoport-Zink space via the uniformization theorem of Rapoport-Zink. 

\subsection{Gross' conjecture} We are led to study the superspecial locus of an orthogonal type Shimura variety by the manuscripts of Gross \cite{Gro-indefinite, Gro-ssp}. Let $(V, Q)$ be a quadratic space over $\QQ_{p}$ whose determinant $\det(V)$ is a unit and whose Hasse invariant $\epsilon(V)=1$. Let $G=\SO(L)$ be the group scheme over $\ZZ_{p}$ defined by a self-dual lattice $L$ of $V$. Let $\mu: \mathbb{G}_{m}\rightarrow G_{\overline{\QQ}_{p}}$ be a suitable miniscule cocharacter and $[b]\in B(G,\mu)$ be the unique basic element in the acceptable part of the Kottwitz set $B(G, \mu)$ defined by \cite[4.5, 4.6]{Rap05}. Given this datum one can attach a basic local Shimura variety $\breve{\mathcal{M}}=\breve{\mathcal{M}}(G, \mu, b)$ which we will refer to as the orthogonal type Rapoport-Zink space, see section $3.1$ for more precise statements. We will write $\breve{\mathcal{M}}_{\mathrm{red}}$ the underlying reduced scheme of $\breve{\mathcal{M}}$. We also need to consider the nearby quadratic space $V^{\prime}$ which is defined to be $V^{\prime}= V^{\Phi}_{K_{0}}$. Here $K_{0}$ is the fraction field of the Witt ring $W_{0}=W(\overline{\FF}_{p})$ and $\Phi$ is the operator given by $b\sigma$ where $\sigma$ is the Frobenius on $K_{0}$.  Then the quadratic space $(V^{\prime}, Q)$ has the same determinant as $V$ but is of Hasse-Witt invariant $\epsilon(V^{\prime})=-1$.  Let $J_{b}=\SO(V^{\prime})$ which acts naturally on $\breve{\mathcal{M}}$. Using the \emph{Bruhat-Tits stratification} for $\breve{\mathcal{M}}$ introduced in \cite{HP17} and \cite{Shen-gen}, we will introduce the \emph{superspecial locus} $\breve{\mathcal{M}}_{\mathrm{ssp}}$ of $\breve{\mathcal{M}}_{\mathrm{red}}$.  Let $K=\QQ_{p}(\sqrt{D})$ be the unramified quadratic extension with $D$ being a unit of $\ZZ_{p}$.  Gross considered the set of \emph{oriented planes} $W$ of discriminant $-D$ and Hasse-Witt invariant $\epsilon(W)=-1$ embedded in $V^{\prime}$. Notice that such an oriented plane $W$ determines a homomorphism $h: T=\SO(W)\rightarrow \SO(V^{\prime})$ and the conjugacy classes of such homomorphisms is parametrized by the $K$-analytic manifold $$X=\SO(V^{\prime})/ \SO(W)\times \SO(W^{\perp}).$$ This is of course the non-archimedean analogue of the conjugacy classes $h: \mathbb{S}\rightarrow G_{\RR}$ where $\mathbb{S}=\mathrm{Res}_{\CC/\RR}\mathbb{G}_{m}$ is the Deligne torus which form part of a Shimura datum. However $\SO(W^{\perp})$ is not always compact and Gross introduced the following modification. Let $W=K.e$ wth $\Nm(e)=p$, then $W$ has a canonical lattice $\calO_{K}.e$ where $\calO_{K}$ is the ring of integers in $K$ and $\mathcal{O}_{K}.e$ provides a natural integral structure of the plane $W$ via its chosen orientation. Let  $$Y= \SO(V^{\prime})/\SO(\calO_{K}.e)\times \SO(M)$$ where $M$ is a lattice in $W^{\perp}$ with stabilizer $\SO(M)$. Therefore $Y$ parametrizes those pairs $(W, M)$ for an oriented plane $W$ and a lattice $M$ in the orthogonal complement $W^{\perp}$. Then $Y$ is an $K$-analytic manifold whose $K$ structure is induced by the orientation on $W$. Since $\epsilon(V^{\prime})=-1$, we let $\Lambda$ be an almost self-dual lattice in $V^{\prime}$. Finally Gross introduced the space $$Z=\SO(V^{\prime})/\SO(\Lambda,\pm)$$ where $\SO(\Lambda,\pm)$ is the stabilizer of $\Lambda$ with a fixed orientation of the quadratic space $\Lambda/ \Lambda^{\vee}$ over $\FF_{p}$.  It is conjectured in \cite[Conjecture 1]{Gro-indefinite} that 
\begin{itemize}
\item $Z$ is the superspecial locus $\breve{\mathcal{M}}_{\mathrm{ssp}}$ of $\breve{\mathcal{M}}_{\mathrm{red}}$.
\item $Y$ is the set of $\mathcal{O}_{K}$-points in $\breve{\mathcal{M}}$ whose reductions lie in $\breve{\mathcal{M}}_{\mathrm{ssp}}$.
\item $X$ is related to the image of the de Rham period map $\pi_{\dR}: \breve{\mathcal{M}}^{\mathrm{rig}}\rightarrow \mathcal{Q}$ where $\mathcal{Q}$ is the flag variety associated to $G$. 
\end{itemize}

The main result of this note is the confirmation of these conjectures.  

\begin{thm}\label{main-RZ}
Let $Def_{\mathrm{ssp}}(\mathcal{O}_{K})$ be the set of points in $\breve{\mathcal{M}}(\mathcal{O}_{K})$ whose reductions lie in $\breve{\mathcal{M}}_{\mathrm{ssp}}$ and let $Def_{\ssp}^{\mathrm{rig}}(K)$ be the rigid analytic $K$-points of $Def_{\mathrm{ssp}}(\mathcal{O}_{K})$.
\begin{enumerate}
\item There is a bijection between $\breve{\calM}_{\mathrm{ssp}}$ and the homogeneous space
$Z$.
\item There is a bijection between $Def_{\mathrm{ssp}}(\mathcal{O}_{K})$ and the homogeneous space $Y$.

\item The image of $\pi_{\dR}$ restricted to $Def^{\mathrm{rig}}_{\mathrm{ssp}}(K)$ is precisely the homogeneous space $X$.
\end{enumerate}
\end{thm}

In the main body of this note, part $(1)$ is  proved in Theorem \ref{thmZ}, part $(2)$ is proved in Corollary \ref{corY} and part $(3)$ is proved in Theorem \ref{thmX}. We also remark that Gross made much more general conjectures about the superspecial loci of Shimura varieties. The method of this note can be applied to confirm his conjectures in the case of Coxeter type Shimura varieties \cite{GH15}. This is indicated in section $4$ of this note. In particular, one can prove Gross' conjecture in the case of unitary Shimura varieties of signature $(n-1, 1)$ by simply using the Bruhat-Tits stratification introduced by Vollaard and Wedhorn \cite{Vol-can10}, \cite{VW-invent11}. This is a simpler case as these unitary Shimura varieties are of PEL type and they do admit convenient moduli interpretations.

\subsection{Applications to Shimura varieties} We will now move to a global set-up and we will denote by $(V, Q)$ a quadratic space $\QQ$ of signature $(n-2, 2)$ with $n\geq 3$. Suppose that the $V$ has unit determinant and Hasse-Witt invariant $1$ at $p$. Let $U=U_{p}U^{p}$ with $U_{p}=G(\ZZ_{p})$ and $U^{p}$ sufficiently small. Similar to the local set-up, we denote by $V^{\prime}$ to be the nearby quadratic space whose signature is $(n, 0)$ and whose determinant is the same as $V$ at $p$ and Hasse-Witt invariant is $-1$.  Then we write $\mathscr{S}_{U, W_{0}}$ over $W_{0}$ the canonical integral model for the Shimura varietiy for the special orthogonal group $G$ \cite{Kisin-abe, MP16}. Let $\mathscr{S}_{\mathrm{ss}}$ and $\mathscr{S}_{\mathrm{ssp}}$ be the supersingular and superspecial locus of $\mathscr{S}_{U, \overline{\FF}_{p}}$. 
 The former one is defined as the basic locus of $\mathscr{S}_{U, \overline{\FF}_{p}}$ under the Newton stratification. The latter one is defined as the minimal Ekedahl-Oort stratum of $\mathscr{S}_{U, \overline{\FF}_{p}}$ constructed in \cite{SZ17}. The Rapoport-Zink uniformization theorem furnishes an isomorphism 
$$I(\QQ)\backslash \breve{\calM}\times G(\mathbb{A}^{p})/ K^{p}\xrightarrow{\sim} (\widehat{\mathscr{S}}_{U, W_{0}})/\mathscr{S}_{\mathrm{ss}}$$
where $I=\SO(V^{\prime})$ is the special orthogonal group defined by the quadratic space $V^{\prime}$ and $(\widehat{\mathscr{S}}_{U, W_{0}})/\mathscr{S}_{\mathrm{ss}}$ is the completion of  ${\mathscr{S}}_{U, W_{0}}$ along the supersingular locus $\mathscr{S}_{\mathrm{ss}}$.  
Then part $(1)$ and $(2)$ of Theorem \ref{main-RZ} imply the following double coset parametrization of the superspecial locus $\mathscr{S}_{\mathrm{ssp}}$ and its $\mathcal{O}_{K}$-lift  $\mathscr{S}^{\ssp}(\mathcal{O}_{K})$ in $(\widehat{\mathscr{S}}_{U, W_{0}})/\mathscr{S}_{\mathrm{ss}}(\mathcal{O}_{K})$. Let $U_{0,p}= \SO(\Lambda, \pm)$,  $U_{1, p}=\SO(\mathcal{O}_{K}.e)\times \SO(M)$. 
\begin{thm}\label{main-shi}
We have the following double cosets parametrizations.
\begin{enumerate}
\item There is a double coset description of the superspecial locus $\mathscr{S}_{\ssp}(\FF_{p^{2}})$ given by
$$ \mathscr{S}_{\ssp}(\FF_{p^{2}})\xrightarrow{\sim}I(\QQ)\backslash I(\mathbb{A}_{f}) /U_{0,p}U^{p}.$$

\item There is a double coset description of $\mathscr{S}^{\ssp}(\mathcal{O}_{K})$ given by
$$\mathscr{S}^{\ssp}(\mathcal{O}_{K})\xrightarrow{\sim}I(\QQ)\backslash I(\mathbb{A}_{f}) /U_{1,p}U^{p}.$$ 
\end{enumerate}
\end{thm}

In the main body of this note, part $(1)$ is proved in Theorem \ref{ssp-shim} and part $(2)$ is proved in Theorem \ref{ssp-shim-lift}.

As an application of the double coset parametrization, we provide a geometric mass formula similar to the simple mass formula for PEL type Shimura varieties given in \cite{Yu11}.  Let $x\in \mathscr{S}_{\mathrm{ssp}}$ be a superspecial point and let $g_{x}\in I(\QQ)\backslash I(\mathbb{A}_{f}) /U_{0,p}U^{p}$ be the corresponding class and we choose a representative in $I(\mathbb{A}_{f}) $ denoted by the same symbol. We define the finite group $\Gamma_{x}$ by the following formula
$$g_{x}Ug^{-1}_{x}\cap I(\QQ)= \Gamma_{x}.$$
We define the arithmetic mass for the superspecial locus $\mathscr{S}_{\ssp}$ using the formula
$$\Mass^{a}_{\mathscr{S}_{\ssp}}=\sum_{x\in \mathscr{S}_{\ssp}}\frac{1}{|\Gamma_{x} |}.$$
We also define the geometric mass  of $\mathscr{S}_{\ssp}$ to be $$\Mass^{g}_{\mathscr{S}_{\ssp}}=\sum_{x\in \mathscr{S}_{\ssp}}\frac{1}{|\Isom_{\mathbf{V}}(x)|}$$
where the definition of $\Isom_{\mathbf{V}}(x)$ is geometric in nature and can be found in section \ref{mass-for} and we only mention that this group is essentially the automorphism group of $x$ preserving Hodge cycles on the abelian varitiety given by $x$.

\begin{thm}Let $\Vol(U)=[I(\widehat{\ZZ}): U]$. We have
\begin{enumerate}
\item $\Mass^{a}_{\mathscr{S}_{\ssp}}=\Mass^{g}_{\mathscr{S}_{\ssp}}$.
\item $\Mass^{a}_{\mathscr{S}_{\ssp}}= \begin{cases}
\Vol(U)\prod^{m}_{r=1}\zeta(1-2r)\frac{1}{2^{m-1}}\frac{p^{2m}-1}{2(p+1)} &\text{if  $n=2m+1$};\\
\Vol(U)\prod^{m}_{r=1}\zeta(1-2r)L(1-m,\chi)\frac{1}{2^{m-1}}\frac{(p^{m-1}+1)(p^{m+1})}{2(p+1)}  &\text{if $n=2m$}.
\end{cases}
$
\end{enumerate}
\end{thm}

\subsection{Notations}Let $p$ be an odd prime and let $\FF$ be an algebraically closed field containing $\FF_{p}$. Let $W_{0}=W(\FF)$ be the Witt ring of $\FF$ and $K_{0}$ be its fraction field. If $M_{1}\subset M_{2}$ are two $W_{0}$-modules, we write $M_{1}\subset^{d} M_{2}$ if the colength of the inclusion is $d$. If $R$ is ring and $L$ is an $R$-module and $R^{\prime}$ is an $R$-algebra, we use the notation $L_{R^{\prime}}=L\otimes_{R} R^{\prime}$.  Let $G$ be a reductive group over $\QQ_{p}$, we denote by $B(G)$ the set of $\sigma$-conjugacy classes in $G(K_{0})$ following Kottwitz \cite{Ko85}.

\subsection{Acknowlegement}The author would like to thank Henri Darmon and Pengfei Guan for supporting his postdoctoral studies. The author would like to thank Benedict Gross for drawing his attention to the conjectures on superspecial locus.

\section{Preliminaries on quadratic spaces}
We begin by reviewing some standard notions in the theory of quadratic spaces and lattices in them. Then we introduce some homogeneous spaces that are conjectured by Gross to be the parametrizing spaces for the superspecial loci of the orthogonal type Shimura varieties and Rapoprot-Zink spaces considered in this note. Everything here can be found in \cite{Gro-indefinite}.

\subsection{Orthogonal spaces}
Let $(V, Q)$ be a quadratic space of dimension $n\geq 3$ over $F$ where $F$ is a general field of characteristic not equal to $2$. Then we can define a symmetric form 
$$[\hspace{1mm},\hspace{1mm}]: V\times V \rightarrow F$$ 
by the formula $[x, y]= Q(x+y)-Q(x)-Q(y)$ which we assume is non-degenerate.  We can choose a basis $\{e_{i}\}^{n}_{i=1}$ such that $Q(\sum^{n}_{i=1}x_{i}e_{i})= a_{1}x^{2}_{1}+a_{2}x^{2}_{2}+\cdots a_{n}x^{2}_{n}.$ We call the product $\det(V)=a_{1}a_{2}\cdots a_{n}$ the \emph{determinant} of $(V, Q)$ and this is an invariant well-defined in $F^{\times}/ F^{\times  2}$. When $\det(V)\in \mathcal{O}^{\times}_{F}F^{\times 2}/ F^{\times 2}$, we say $V$ has unit determinant. Let $a_{i}, a_{j}$ be two classes in $\mathrm{H}^{1}(F, \mu_{2})=F^{\times}/ F^{\times 2}$ and $(a_{i}, a_{j})\in \mathrm{H}^{2}(F, \mu_{2})$ be their cup product in $\mathrm{H}^{2}(F, \mu_{2})=Br_{2}(F)$ where $Br_{2}(F)$ is the two torsion subgroup of the Brauer group of $F$. Then the \emph{Hasse-Witt invariant} of $V$ is defined as $\epsilon(V)=\sum_{i< j}(a_{i}, a_{j})\in Br_{2}(F)$.  Suppose that $(V, Q)$ is a two dimensional quadratic space then an \emph{orientation} of $(V, Q)$ is an embedding $\iota: K\rightarrow \End(V)$ such that $Q(\iota(\alpha)x)=\Nm(\alpha)Q(x)$ for the algebra $K=F[t]/(t^{2}+\det(V))$. The following lemma is well known. 

\begin{lemma}\label{Orientation}
An orientation for $(V, Q)$ is equivalent to a choice of an isotropic line in $(V_{K}, Q)$. 
\end{lemma}

Suppose now $F=F_{v}$ is a local field. Then we distinguish the following situations
\begin{itemize}
\item $F_{v}=\mathbb{C}$, then $(V, Q)$ is determined by its dimension;
\item $F_{v}=\mathbb{R}$, then $(V, Q)$ is determined by its signature;
\item $F_{v}$ is non-archimedean, then $(V, Q)$ is determined by its dimension, determinant and its Hasse-Witt invariant.
\end{itemize}

\subsection{Orthogonal lattices}
We assume now $F$ is a local non-archimedean field with ring of integers $\calO_{F}$ and uniformizer $\pi_{F}$. The residue field will be denoted as $k$. 
As usual we assume that the residue characteristic is not $2$.  An orthogonal lattice $L\subset V$ is a free $\calO_{F}$-module that spans $V$, we denote by $$L^{\vee}=\{x\in V: [x, L]\in \calO_{F}\}$$
the integral dual of $L$.  We say $L$ is \emph{self-dual} if $L=L^{\vee}$ and for $L$ to be self-dual it is necessary that $\det(V)\in \calO^{\times}_{F}F^{\times 2}/ F^{\times2 }$ is a unit and the Hasse-Witt invaraint is $1$. We say $L$ is \emph{almost-self-dual} if $L/L^{\vee}$ is a two dimensional quadratic space over the residue field $\calO_{F}/\pi_{F}$, which is isomorphic as an orthogonal space to the quadratic extension of $\calO_{F}/\pi_{F}$ with its norm form. The following theorem summarizes when a self-dual or an almost self-dual lattice exists in a quadratic space.

\begin{thm}\label{lattice}
Let $(V, Q)$ be a quadratic space over $F$ of dimension $n$ whose determinant is a unit.
\begin{itemize}
\item If $\epsilon(V)=1$, then $V$ contains a self-dual lattice $L$. There is a transitive action of the special orthogonal group $\SO(V)(F)$ on these lattices.
\item If $\epsilon(V)=-1$, then $V$ contains an almost-self-dual lattice $L$. There is a transitive action of the special orthogonal group $\SO(V)(F)$ on these lattices.
\end{itemize}
\end{thm}

\begin{proof}
This is well-known and see \cite[Theorem 6.1]{Gro-indefinite} for a complete proof.
\end{proof}

The stabilizer $H(L)$ of such a lattice is a parahoric subgroups of $\SO(V)(F)$:
\begin{itemize}
\item if $\epsilon=1$, then $H(L)$ is a hyperspecial subgroup of $\SO(V)(F)$;
\item if $\epsilon=-1$, then $H(L)$ is a parahoric subgroup of the non-quasi-split group $\SO(V)(F)$.
\end{itemize}

\subsection{Some homogeneous spaces}

As a motivational example, consider first the Hermitian symmetric space $X$ associated to the group $\SO(V)$ for $V$ a real quadratic space of signature $(n-2, 2)$ and recall we always assume that $n\geq 3$. Then $X$ is the set of oriented negative definite two planes $W$ in $V$ which is a complex manifold of dimension $n-2$ with a transitive action of $\SO(V)$. Now an oriented negative definite plane $W$ in $V$ gives rise to a homomorphism $h: \SO(W)\rightarrow \SO(V)$ over $\mathbb{R}$ whose centralizer is $H=\SO(W)\times \SO(W^{\perp})$ and therefore we have $X=\SO(V)/ H$. The tangent space at identity is given by $\Hom(W, W^{\perp})$ which has a complex Hermitian structure coming from the chosen orientation of $W$. 

We now consider the case $(V,Q)$ is a quadratic space over a local non-archimedean field $F$ with $\epsilon(V)=-1$ and unit determinant. Let $K=F(\sqrt{D})$ be the unramified quadratic extension of $F$. Gross introduced the following homogeneous spaces and conjectured their relation with the superspecial locus.
\begin{itemize}

\item The space $X=\SO(V)/\SO(W)\times \SO(W^{\perp})$: We let $X$ be the set of  oriented planes $W$ in $V$ with fixed determinant $-D$ and $\epsilon(W)=-1$. Then exactly as the real case we have $X=\SO(V)/H$ where $H=\SO(W)\times \SO(W^{\perp})$.

 \item The space $Y= \SO(V)/\SO(\calO_{K}.e)\times \SO(M)$: However $\SO(W^{\perp})$ is not always compact and Gross introduced the following modification. Given the orientation on $W$, we can write $W=K.e$ wth $\Nm(e)=\pi_{F}$. Then $W$ has a canonical lattice $\calO_{K}.e$. Let  $Y= \SO(V)/\SO(\calO_{K}.e)\times \SO(M)$ where $\calO_{K}.e$ provides a natural integral structure of the plane $W$ via its chosen orientation and $M$ is a self-dual lattice in $W^{\perp}$. Therefore $Y$ parametrizes those pairs $(W, M)$ for an oriented two plane $W$ and a lattice $M$ in its orthogonal complement $W^{\perp}$. Then $Y$ is an $K$-analytic manifold whose $K$ structure is induced by the orientation on $W$.

\item The space $Z=\SO(V)/\SO(\Lambda,\pm)$: Since $\epsilon(V)=-1$, we let $\Lambda$ be an almost self-dual lattice in $V$. Another homogeneous space we will consider is $Z=\SO(V)/\SO(\Lambda,\pm)$ where $\SO(\Lambda,\pm)$ is the stabilizer of $\Lambda$ with a fixed orientation on $\Lambda/ \Lambda^{\vee}$. 
\end{itemize}

There is a natural map
$$\red: Y\rightarrow Z.$$
The map $\red$ sends $(W, M)$ to $\Lambda=M+\calO_{K}.e$ with the orientation of $\Lambda/\Lambda^{\vee}$ induced by the orientation of $W$.  The fibers of this map $\red^{-1}(z)$ with $z\in Z$ are quite simple and can be shown to be polydiscs. The following proposition can be proved using the Moy-Prasad filtration of the parahoric subgroup $\SO(\Lambda,\pm)$. See \cite[Theorem 9.1]{Gro-indefinite}
\begin{proposition}
The fibers $\red^{-1}(z)$ of $\red: Y\rightarrow Z$ are polydiscs of the form $(\pi\calO_{K})^{n-2}$. 
\end{proposition}

We will show later that the homogeneous space $Z$ can be used to parametrize the superspecial locus on the orthogonal type Rapoport-Zink space while $Y$ can be used to parametrize the $\mathcal{O}_{K}$-lift of the superspecial locus. Finally $X$ will be related to the admissible locus of the period domain.

\section{Bruhat-Tits stratification of orthogonal type Rapoport-Zink spaces}
In this section we review the theory of Howard and Pappas on the Bruhat-Tits stratification of the basic Rapoport-Zink space of $\Gspin$ type. In this section we fix an algebraic closure $\FF$ of $\FF_{p}$ and write $W_{0}=W(\FF)$ its ring of Witt vectors. We denote by $K_{0}=W(\FF)[\frac{1}{p}]$ the fraction field of $W_{0}$.

\subsection{Local Shimura varieties of orthogonal type} We fix a quadratic space $(V, Q)$ over $\QQ_{p}$ and we assume that $\det(V)$ is a unit and $\epsilon(V)=1$ throughout this section. By Theorem \ref{lattice}, we have a self-dual lattice $L \subset V$. For any $\ZZ_{p}$-algebra $R$, we write the base change of $L$ to $R$ by $L_{R}$. We denote by $C(L)$ the Clifford algebra of $L$ which is a $\mathbb{Z}/2\mathbb{Z}$-graded $\mathbb{Z}_{p}$-algebra with the grading $C(L)=C^{+}(L)\oplus C^{-}(L)$. It is a free rank $2^{n}$-algebra over $\ZZ_{p}$ equipped with a canonical involution $*: C(L)\rightarrow C(L)$ and a reduced trace map $\Tr: C(L)\rightarrow \ZZ_{p}$. We define the spinor similitude group $G^{\Diamond}$ over $\ZZ_{p}$ to be the group whose $R$ points for any $\ZZ_{p}$-algebra $R$ is given by
$$G^{\Diamond}(R)=\{g\in C^{+}(L_{R})^{\times}: gL_{R}g^{-1}=L_{R}, gg^{*}\in R^{\times}\}.$$
We define a character $\eta_{G^{\Diamond}}: G^{\Diamond}\rightarrow \mathbb{G}_{m}$ by $g\rightarrow gg^{*}$, this is called the spinor norm. Picking an element $\delta^{*}=-\delta$ and let $\psi_{\delta}: C(L)\times C(L)\rightarrow \ZZ_{p}$ be the symplectic form given by $\psi_{\delta}(c_{1},c_{2})=\Tr(c_{1}\delta c^{*}_{2})$. The natural left action of $C(L)^{\times}$ on $C(L)$ gives a closed immersion 
\begin{equation}
i: G^{\Diamond}\rightarrow \GSp(C(L), \psi_{\delta}).
\end{equation}
In fact, by \cite{Kisin-abe}, there is a finite list of tensors $\{s_{\alpha}\}\subset C(L)^{\otimes}$ such that $G^{\Diamond}$ is precisely the stabilizer in $\GL(C(L)))$ of this list. We will always fix such a list. 
There are  short exact sequences
$$1\rightarrow \mathbb{G}_{m}\rightarrow G^{\Diamond}\rightarrow \SO(L)\rightarrow 1$$
and
$$1\rightarrow \mu_{2}\rightarrow G^{\Diamond \der}\rightarrow \SO(L)\rightarrow 1$$
with the derived group $G^{\Diamond \der}$ being simply connected.  We will set $G=\SO(L)$ from here on.

Following Howard-Pappas \cite{HP17}, we choose a basis $x_{1}, x_{2}, \cdots, x_{n}\in L$ such that  the matrix representing the inner product is
$$
\begin{pmatrix}
0 & 1 & & & &\\
1 & 0 & & & &\\
& & * & & & &\\
& &   &*& & &\\
& &   & &\ddots & &\\
& & & & & &*\\
\end{pmatrix}.
$$
This defines 
\begin{itemize}
\item a cocharacter $\mu^{\Diamond}: \mathbb{G}_{m}\rightarrow G^{\Diamond}$ by $\mu^{\Diamond}(t)=t^{-1}x_{1}x_{2}+x_{2}x_{1}$;
\item an element $b^{\Diamond}=x_{3}(p^{-1}x_{1}+x_{2})\in G^{\Diamond}(\QQ_{p})$.
\end{itemize}

The element $b^{\Diamond}$ gives a basic element in the Kottwitz set $B(G^{\Diamond})$ defined in \cite{Ko85} and it in fact lies in set of the neutral acceptable elements $B(G^{\Diamond}, \mu^{\Diamond})$. The composition $i\circ\mu:\mathbb{G}_{m}\rightarrow \GSp(C(L), \psi_{\delta}) $ is a miniscule cocharacter. This means that
the datum $(G^{\Diamond}, \{\mu^{\Diamond}\}, b^{\Diamond}, i)$ is a local (unramfied) Shimura datum of Hodge type in the sense of \cite{HP17}. One then attaches a local Shimura variety which in this case we will be referred to as the $\GSpin$-type Rapoport-Zink space: $\breve{\calM}^{\Diamond}=\breve{\calM}(G^{\Diamond}, \{\mu^{\Diamond}\}, b^{\Diamond})$. This is a formal scheme locally formally of finite type over $\Spf (W_{0})$. The image $(G, \{\mu\}, b)$ of $(G^{\Diamond}, \{\mu^{\Diamond}\}, b^{\Diamond})$ under the natural map $G^{\Diamond}\rightarrow G$ is local Shimura datum of abelian type in the sense of Shen \cite[Definition 4.1]{Shen-gen} and we obtain the orthogonal type Rapoport-Zink space $\breve{\calM}=\breve{\calM}(G, {\mu}, b)$ and by \cite[Corollary 7.8]{Shen-gen} we in fact have \begin{equation}\label{spin-to-or}\breve{\calM}=\breve{\calM^{\Diamond}}/p^{\ZZ}.\end{equation} 

We summarize the above discussion in the following diagram of Rapoport-Zink spaces  
\begin{equation}\label{spin-ortho-diag}
\begin{tikzcd}
 \breve{\mathcal{M}}^{\Diamond}\arrow[r, hook] \arrow[d]
    & \breve{\mathcal{M}}_{\GSp(C(L))}\\
   \breve{\mathcal{M}} & 
\end{tikzcd}
\end{equation} 
here  $\breve{\mathcal{M}}_{\GSp(C(L))}$ is the basic Rapoport-Zink space for the group $\GSp(C(L))$ and the horizontal arrow is given by the embedding of the local Shimura datum $(G^{\Diamond}, \{\mu^{\Diamond}\}, b^{\Diamond}, i)$ in the corresponding  local Shimura datum for the group  $\GSp(C(L))$, the vertical map is the natural covering map. We will need the following lemma later. 

\begin{lemma}\label{X_{0}}
There is a unique $p$-divisible group $X_{0}$ over $\FF$ such that its covariant {\Dieu} module $\mathbb{D}(X_{0})(W_{0})=C(L)_{W_{0}}$ with Frobenius induced by $b^{\Diamond}\sigma$ and whose Hodge filtration $V \mathbb{D}(X_{0})(\FF)\subset \mathbb{D}(X_{0})(\FF)$ is induced by the reduction of ${\mu}^{\Diamond}_{\FF}$ of $\mu^{\Diamond}$.
\end{lemma}
\begin{proof}
This is proved in \cite[Lemma 2.2.5]{HP17} but note here we are using covariant {\Dieu} theory.
\end{proof}

\begin{remark}
We in fact can and will assume that $X_{0}$ in the previous lemma comes from a global point $x_{0}$ on the integral model of the $\GSpin$-Shimura variety $\mathscr{S}^{\Diamond}_{K^{\Diamond}}$ that we will introduce later. Also we note that the construction of the Rapoport-Zink space $\breve{\mathcal{M}}^{\Diamond}$ is by global method and uses the existence of the integral model $\mathscr{S}^{\Diamond}_{K^{\Diamond}}$. 
\end{remark}
The formal scheme $\breve{\calM}^{\Diamond}$ represents a functor $$\breve{\calM}^{\Diamond}:  \Nilp_{W_{0}}\rightarrow \Sets$$ where $\Nilp_{W_{0}}$ is the category of all Noetherian $W_{0}$-algebra in which $p$ is locally nilpotent. This functor is not easy to describe. But  let $\ANilp_{W_{0}}$ be the full subcategory of Noetherian $W_{0}$-algebra in which $p$ is nilpotent. We restrict this functor to $\ANilp_{W_{0}}$ then this functor assigns each $R\in \ANilp_{W_{0}}$ the set of isomorphism classes of $(X, \rho, (t_{\alpha}))$ where

\begin{itemize}
\item $X$ is a $p$-divisible group over $R$;
\item $\rho: X_{0}\times_{\FF} \bar{R}\rightarrow X\times_{R} \bar{R}$ is a quasi-isogeny;
\item $(t_{\alpha})$ is a collection of tensors in $\mathbb{D}(X)^{\otimes}(R)$ such that the sheaf $$\Isom_{t_{\alpha}, s_{\alpha}\otimes 1}(\mathbb{D}(X)(R), C(L)\otimes R)$$ of isomorphisms that respect the tensors $(t_{\alpha})$ and $(s_{\alpha}\otimes 1)$  and is compatible with filtration induced by $\mu$ on $C(L)\otimes R$ and the Hodge filtration on $\mathbb{D}(X)(R)$ is a $P_{\mu}\times_{W_{0}} R$ torsor where $P_{\mu}$ is the parabolic induced by $\mu$.
\end{itemize}

\subsection{The group $J_{b}(\QQ_{p})$} The Rapoport-Zink space $\breve{\calM}^{\Diamond}$ carries an action of the group $J_{b^{\Diamond}}(\QQ_{p})$ whose definition we now recall. For an $\QQ_{p}$-algebra $R$, the $R$ points of the algebraic group $J_{b^{\Diamond}}$ is given by
$$\{g\in G^{\Diamond}(R\otimes_{\QQ_{p}}K_{0}): gb^{\Diamond}\sigma(g)^{-1}=b^{\Diamond}\}.$$

Let $\Phi=b^{\Diamond}\sigma$ which acts on $V_{K_{0}}$, consider the quadratic space $(V^{\Phi}_{K_{0}}, Q)$ given by restricting the space $(V_{K_{0}}, Q)$. Then $(V^{\Phi}_{K_{0}}, Q)$ has the same determinant as $(V,Q)$ but $\epsilon(V^{\Phi}_{K_{0}})=-\epsilon(V)=-1$ by \cite[Proposition 4.25]{HP17}. Now the group $J_{b^{\Diamond}}(\QQ_{p})$ can be explicitly described using the space $V^{\Phi}_{K_{0}}$.

\begin{lemma}\label{J-diamond}
The group $J_{b^{\Diamond}}(\QQ_{p})=\GSpin(V^{\Phi}_{K_{0}})$.
\end{lemma}
\begin{proof}
See the proof of \cite[Proposition 4.2.5]{HP17}.
\end{proof}

Similarly, $\breve{\calM}$ admits an action of $J_{b}(\QQ_{p})=\{g\in G(K_{0}): gb\sigma(g)^{-1}=b\}.$

\begin{lemma}
The group $J_{b}(\QQ_{p})=\SO(V^{\Phi}_{K_{0}})$.
\end{lemma}

\begin{proof}
This follows immediately from Lemma \ref{J-diamond}.
\end{proof}

Recall the $p$-divisible group $X_{0}$ introduced in Lemma \ref{X_{0}}. That lemma also implies the following. Since $$\End_{W_{0}}(\mathbb{D}(X_{0})(W_{0}))=\Hom_{W_{0}}(\mathbb{D}(X_{0})(W_{0}), \mathbb{D}(X_{0})(W_{0}))=C(L)_{W_{0}}^{*}\otimes C(L)_{W_{0}},$$ in particular, writing $D(X_{0})=\mathbb{D}(X_{0})(W_{0})[\frac{1}{p}]$ we have 
\begin{equation}
\label{V-action}V_{K_{0}}\rightarrow \End({D}(X_{0})).
\end{equation} 
The action of $\Phi=b^{\Diamond}\circ\sigma$ on $V_{K_{0}}$ is related to the induced Frobenius $F$ on the isocrystal $D(X_{0})$ by $\Phi(x)=F\circ x \circ F^{-1} \in  \End({D}(X_{0}))$. This implies that we have a map
\begin{equation*}
V^{\Phi}_{K_{0}}\rightarrow \End(D(X_{0}), F)=\End(X_{0})_{\QQ}.
\end{equation*}

\subsection{Bruhat-Tits stratification} A \emph{vertext lattice} $\Lambda$ in $V^{\Phi}_{K_{0}}$ is by definition a $\ZZ_{p}$-lattice such that \begin{equation}p\Lambda\subset\Lambda^{\vee}\subset \Lambda.\end{equation} We define the type $t_{\Lambda}$ of $\Lambda$ by setting $t_{\Lambda}=\dim_{\FF_{p}} \Lambda/\Lambda^{\vee}$. This is an even integer such that $2\leq t(\Lambda)\leq t_{\max}$ where $t_{\max}$ is given by
\begin{itemize}
\item $t_{\max}=n-1$ if $n$ is odd;
\item $t_{\max}=n-2$ if $n$ is even and $\det(V)=(-1)^{\frac{n}{2}}$;
\item $t_{\max}=n$ if $n$ is even and $\det(V)\neq(-1)^{\frac{n}{2}}$.
\end{itemize}

 If $t_{\Lambda}=r$, we will also write $p\Lambda\subset^{n-r}\Lambda^{\vee}\subset^{r} \Lambda$. We will be particularly concerned with those lattices of type $t_{\Lambda}=2$ as these will be related to the superspecial points on $\breve{\mathcal{M}}$. First of all, this is a non-empty collection of lattices. 
 
 \begin{lemma} 
 The collection of lattices $\Lambda$ in $V^{\Phi}_{K_{0}}$ whose $t_{\Lambda}=2$ is non-empty. 
 \end{lemma}
 \begin{proof}
 Consider the quadratic space $(V^{\Phi}_{K_{0}}, Q)$. Then by Theorem \ref{lattice}, there is an almost self-dual lattice $\Lambda^{\vee}$ in $V^{\Phi}_{K_{0}}$. Then $\Lambda/\Lambda^{\vee}$ is a quadratic space over $\FF_{p}$ of dimension $2$. Finally $p$ annihilates $\Lambda/\Lambda^{\vee}$ and therefore $p\Lambda\subset \Lambda^{\vee}$.
 \end{proof}
 
 Fix a vertex lattice $\Lambda^{\vee}\subset\Lambda\subset V^{\Phi}_{K_{0}}$, we denote by $\breve{\calM}^{\Diamond}_{\Lambda}$ the closed formal subscheme defined by those points $(X, \rho, (t_{\alpha}))$ that satisfies the condition \begin{equation}\label{MLam}\rho\circ\Lambda^{\vee}\circ\rho^{-1}\subset \End(X).\end{equation}
 Then it follows from \eqref{spin-to-or} that it makes sense to put
 \begin{equation}\label{ortho-spin}
 \breve{\calM}_{\Lambda}=\breve{\calM}^{\Diamond}_{\Lambda}/p^{\ZZ}.
 \end{equation} 
 We will refer to this as the \emph{vertex stratum} associated to $\Lambda$ in $\breve{\calM}$. We will later need to use the \emph{Bruhat-Tits stratum} defined by 
\begin{equation*}
\BT_{\Lambda}=\breve{\calM}_{\Lambda}-\bigcup_{\Lambda^{\prime}\subset \Lambda}\breve{\calM}_{\Lambda^{\prime}}.
\end{equation*}
 Next we recall that the $\FF$ points of $\breve{\calM}_{\Lambda}$ can be described by the so-called special lattices.  A \emph{special lattice} $L\subset V_{K_{0}}$ is a self-dual $W_{0}$ lattice of $V_{K_{0}}$ such that 
 \begin{equation*}
 L+\Phi_{*}(L)/L\cong W_{0}/pW_{0} 
 \end{equation*}  
 where $\Phi_{*}(L)$ is the $W_{0}$-module generated by $\Phi(L)$. If $L$ is a special lattice, then denote by $L^{(r)}=L+\Phi_{*}(L)+\cdots+\Phi^{r}_{*}(L)$.  The following theorem implies that the chain of lattices $L\subset L^{(1)}\subset \cdots \subset L^{(r)}$ stabilizes and gives rise to a vertex lattice containing $L$. 
 
 \begin{proposition}\label{special-vertex}
 There is a unique integer $1\leq d\leq t_{\max}/2$ such that $L^{(d)}=L^{(d+1)}$. Moreover $L^{(d)}=\Lambda(L)_{W_{0}}$ for a vertex lattice $\Lambda(L)$ of type $2d$.
 \end{proposition}
\begin{proof}
This is proved in \cite[Proposition 5.2.2]{HP17}.
\end{proof}

For simplicity, we put $C=C(L)$. Let $y\in \breve{\calM^{\Diamond}}(\FF)$, we denote by $X_{y}$ the $p$-divisible group at $y$, then we set
\begin{equation*}\label{M-Dieu}
\begin{split}
 &M_{y}=\mathbb{D}(X_{y})(W_{0});\\
& M_{1, y}= \Fil^{1} M_{y}= V M_{y}.\\
 \end{split}
\end{equation*}
We let $D(X_{y})=\mathbb{D}(X_{y})(W_{0})[\frac{1}{p}]$. Recall there is an action of $V_{K_{0}} \rightarrow \End({D}(X_{y}))$ defined in \eqref{V-action}.  Then we put 

\begin{equation}\label{special-lattice}
\begin{split}
L_{y} &=\{z\in V_{K_{0}}: zM_{1, y}\subset M_{1,y}\};\\
L^{\sharp}_{y} & =\{z\in V_{K_{0}}: zM_{y}\subset M_{y}\};\\
L^{\sharp\sharp}_{y} &=\{z\in V_{K_{0}}: zM_{1,y}\subset M_{y}\}.\\
  \end{split}
\end{equation}

\begin{theorem}
For any point $y\in \breve{\calM^{\Diamond}}(\FF)$, the lattice $L_{y}$ is a special lattice in $V_{K_{0}}$. Moreover $L^{\sharp}_{y}=\Phi_{*}(L_{y})$ and $L^{\sharp}_{y}+L_{y}=L^{\sharp\sharp}_{y}$. 
The association $y\rightarrow L_{y}$ gives rise to bijections 
 $$\breve{\calM}(\FF)=\breve{\calM^{\Diamond}}(\FF)/p^{\ZZ}= \{\text{specical lattices }  L  \text{ in } V_{K_{0}}\};$$
 $$\breve{\calM}_{\Lambda}(\FF)=\breve{\calM^{\Diamond}_{\Lambda}}(\FF)/p^{\ZZ}= \{\text{specical lattices }  L  \text{ in } V_{K_{0}}: \Lambda^{\vee}_{W_{0}}\subset L \subset \Lambda_{W_{0}}\}.$$
\end{theorem}

\begin{proof}
This is proved in \cite[Proposition 6.2.2]{HP17} combined with \eqref{ortho-spin}. 
\end{proof}

The schemes $\breve{\calM}_{\Lambda}$ are related to the Deligne-Lusztig varieties. More precisely consider the quadratic space $\Omega=\Lambda_{W_{0}}/ \Lambda_{W_{0}}^{\vee}$ over $\FF$ equipped with Frobenius operator $\Phi$. Let $S_{\Lambda}$ be the closed subscheme of the orthogonal grassmanian $\mathrm{OGr}(\Omega)$  whose $\FF$ points are described by
$$S_{\Lambda}(\FF)=\{\text{Lagrangians }\mathcal{L}\subset \Omega: \dim_{\FF} \mathcal{L}+\Phi(\mathcal{L})=\frac{t_{\Lambda}}{2}+1\}.$$
Sending $\Lambda^{\vee}_{W_{0}}\subset L\subset \Lambda_{W_{0}}$ to $L/\Lambda^{\vee}_{W_{0}}\subset \Lambda_{W_{0}}/ \Lambda^{\vee}_{W_{0}}$ in fact gives rise to a bijection
$$\breve{\calM}_{\Lambda}(\FF)\xrightarrow{\cong} S_{\Lambda}(\FF).$$

We choose a basis $\{e_{1}, e_{2}, \cdots e_{d}, f_{1}, f_{2}\cdots, f_{d}\}$ such that $\mathrm{Span}_{\FF}\{e_{1}, e_{2}, \cdots, e_{d}\}$ and $\mathrm{Span}_{\FF}\{f_{1}, \cdots, f_{d}\}$ are both totally isotropic such that the Frobenius fixes $\{e_{1}, \cdots, e_{d-1}\}$ and $\{f_{1}, \cdots, f_{d-1}\}$ and swaps $e_{d}$ and $f_{d}$.  Let $$\mathcal{F}_{d}^{+}=\mathrm{Span}_{\FF}\{e_{1}, \cdots, e_{d}\}$$ and 
$$\mathcal{F}_{d}^{-}=\mathrm{Span}_{\FF}\{e_{1}, \cdots, e_{d-1}, f_{d}\}.$$
Let $P^{+}$ and $P^{-}$ be the corresponding parabolic subgroups of $\SO(\Omega)$. Then $S_{\Lambda}=X_{P^{+}}(1)\sqcup X_{P^{-}}(1)$
is a disjoint union of generalized Deligne-Lusztig varieties. Note that when $\Omega$ has dimension $2$, then $P^{+}$ and $P^{-}$ are Borel subgroups of $\SO(\Omega)$. And $1$ is a Coxeter element in the Weyl group $W=N(T)/T$ for the torus determined by the chosen basis. In this case both $X_{P^{+}}(1)$ and  $X_{P^{-}}(1)$ consist of a single point, moreover the Frobenius action on $S_{\Lambda}$ swaps these two points. 

 \subsection{Bruhat-Tits strata of type $2$} Let $\Lambda$ be a vertex lattice of type $2$ in $V^{\Phi}_{K_{0}}$. Then we claim those $\Lambda$ are precisely the almost-self-dual lattices in $V^{\Phi}_{K_{0}}$.
 \begin{lemma}
The collection of vertex lattices of type $2$ is in bijection with almost-self-dual lattices in $V^{\Phi}_{K_{0}}$.
 \end{lemma}
 \begin{proof}
 It is clear that an almost-self-dual lattices in $V^{\Phi}_{K_{0}}$ gives rise to a vertex lattice of type $2$. Conversely if $\Lambda$ is vertex lattice of type $2$, then $\Lambda/\Lambda^{\vee}$ does not have an isotropic line and this implies that $\Lambda/\Lambda^{\vee}$ is isomorphic to the quadratic extension $\FF_{p^{2}}/\FF_{p}$. In fact, if $l\subset \Lambda/\Lambda^{\vee}$ is an isotropic line, then $l+\Lambda^{\vee}$ is a vertex lattice of type $0$ which is a contradiction to $2\leq t_{\Lambda}\leq t_{max}$. This implies that $\Lambda$ is an almost-self-dual lattice.
 \end{proof}
 
\begin{definition}\label{ssp-locus}
We define the superspecial locus $\breve{\calM}_{\ssp}$  of $\breve{\calM}$  as the union of all the vertex strata of type $2$
$$ \breve{\calM}_{\ssp}=\bigcup_{\Lambda, t_{\Lambda}=2} \breve{\calM}_{\Lambda}.$$
\end{definition}

\begin{lemma}
The superspecial locus $\breve{\calM}_{\ssp}$ is defined over $\FF_{p^{2}}$. 
\end{lemma}
\begin{proof}
Recall  that $$\breve{\calM}_{\Lambda}(\FF)=\breve{\calM^{\Diamond}_{\Lambda}}(\FF)/p^{\ZZ}= \{\text{specical lattices }  L  \text{ in } V_{K_{0}}: \Lambda^{\vee}_{W_{0}}\subset L \subset \Lambda_{W_{0}}\}.$$ Let $\Lambda$ be vertex lattice of type $2$. Let $y\in \breve{\calM}_{\Lambda}$ be a superspecial point. Since $L=L_{y}$ is by definition self-dual and $\Lambda$ is almost-self-dual, we have $$\Lambda^{\vee}_{W_{0}}\subset^{1} L\subset^{1}\Lambda_{W_{0}}.$$  In particular, $L/\Lambda^{\vee}_{W_{0}}$ gives rise to an isotropic line in the two dimensional quadratic space $\Lambda_{W_{0}}/\Lambda^{\vee}_{W_{0}}$. It is well known that there are two such isotropic lines in  $\Lambda_{W_{0}}/\Lambda^{\vee}_{W_{0}}$.  Using the fact that $\Lambda$ is $\Phi$-invariant, we have also $\Lambda^{\vee}_{W_{0}}\subset^{1} \Phi_{*}(L)\subset^{1}\Lambda_{W_{0}}.$ Moreover since $\Phi_{*}(L)+\Phi^{2}_{*}(L)=L+\Phi_{*}(L)$ and $\dim L+\Phi_{*}(L)/\Phi_{*}(L)=1$, $\Phi_{*}(L)$ is also special and thus belong to $\breve{\calM}_{\Lambda}(\FF)$. Note that $\Phi^{2}_{*}(L)$ is also special lattice and $\Phi^{2}(L)\subset^{1}\Phi_{*}(L)+\Phi^{2}_{*}(L)=L+\Phi_{*}(L)$. We conclude that $L=\Phi^{2}_{*}(L)$ or $L=\Phi_{*}(L)$. Since $L$ is special and thus not $\Phi$-invariant, $L=\Phi^{2}_{*}(L)$. It then follows that $L/\Lambda^{\vee}_{W_{0}}$ is fixed by $\Phi^{2}$ and therefore $\breve{\calM}_{\ssp}$ is defined over $\FF_{p^{2}}$. 
\end{proof}

Recall we have the bijection $$\breve{\calM}_{\Lambda}(\FF)\xrightarrow{\cong}S_{\Lambda}(\FF)=\{\text{Lagrangians subspace  } \mathcal{L}\subset \Omega: \dim_{\FF}\mathcal{L}+\Phi(\mathcal{L})=2\}.$$ The above Lemma shows that there are two special lattices $L$ and $\Phi_{*}(L)$ in $\breve{\calM}_{\Lambda}$.  Correspondingly there are only two points in $S_{\Lambda}(\FF)$ given by $L/\Lambda^{\vee}_{W_{0}}$ and $\Phi_{*}(L)/\Lambda^{\vee}_{W_{0}}$.  We will write $\breve{\calM}_{\Lambda}(\FF) =\{x_{\circ}, x_{\bullet}\}$ with $L_{x_{\circ}}=\Phi_{*}(L_{x_{\bullet}})$ and $L_{x_{\bullet}}=\Phi_{*}(L_{x_{\circ}})$. 

\subsection{Parametrization of the superspecial locus}

Now we are ready to relate $\breve{\calM}_{\ssp}$ to the homogeneous space $Z$ introduced in \S 2.3.  We first point out why the space $\breve{\mathcal{M}}_{\ssp}$ introduced in Definition \ref{ssp-locus} deserves to be called the superspecial locus. 
\begin{definition}
Let $X$ be a $p$-divisible group and $\mathbb{D}(X)(W_{0})$ be its covariant {\Dieu} module. Then we call $X$ superspecial if $$F^{2}\mathbb{D}(X)(W_{0})=p\mathbb{D}(X)(W_{0}).$$ 
\end{definition}

\begin{lemma}
The set $\breve{\calM}_{\ssp}(\FF)$ consists of those points $y\in\breve{\calM}(\FF)$ such that the $p$-divisible group $X_{y}$ associated to $y$ is superspecial. 
\end{lemma}

\begin{proof}
Consider the special lattice $L^{\sharp}_{y} =\{z\in V_{K_{0}}: zM_{y}\subset M_{y}\}$.  Suppose that $y\in \breve{\calM}_{ssp}(\FF)$, then $\Phi^{2}(L^{\sharp}_{y})=L^{\sharp}_{y}$. We have the relation $\Phi(z)=F\circ z\circ F^{-1}$ for $z\in V_{K_{0}}$ in $\End(M_{y})$. This implies by definition that 
$\Phi^{2}(L^{\sharp}_{y}) =\{z\in V_{K_{0}}: zF^{2}M_{y}\subset F^{2}M_{y}\}$. On the other hand, $\Phi^{2}(L^{\sharp}_{y})=L^{\sharp}_{y}$ implies that $\Phi^{2}(L^{\sharp}_{y}) =\{z\in V_{K_{0}}: zM_{y}\subset M_{y}\}$. This implies that $F^{2}M_{y}=p^{m}M_{y}$ for some $m\geq0$. But we know $X_{y}$ has slope $1/2$. This implies that $m=1$. This shows that for $y\in \breve{\calM}_{\ssp}(\FF)$, $X_{y}$ is superspecial.

The above argument also shows that $X_{y}$ is superspecial if and only if $\Phi^{2}(L^{\sharp}_{y})=L^{\sharp}_{y}$. Notice by Proposition \ref{special-vertex},  $\Phi^{2}(L^{\sharp}_{y})=L^{\sharp}_{y}$ if and only if $y\in \breve{\calM}_{\ssp}(\FF)$. This finishes the proof.
\end{proof}

\begin{remark}
By the previous lemma, the superspecial locus $\breve{\calM}_{\ssp}$ of $\breve{\calM}$ is the image of the restriction of the usual superspecial locus of $\breve{\calM}_{\GSp(C(L))}$ to the $\breve{\mathcal{M}}^{\Diamond}$. See the diagram in \eqref{spin-ortho-diag}.
\end{remark}

\begin{theorem}\label{thmZ}
There is a bijection between $\breve{\calM}_{\ssp}(\FF)$ and the homogeneous space
$$Z=\SO(V^{\Phi}_{K_{0}})/\SO(\Lambda_{0}, \pm) $$
where $\Lambda_{0}$ is a fixed type $2$ vertex lattice.
\end{theorem}

\begin{proof}
One can find an element $g\in \SO(V^{\Phi}_{K_{0}})$ such that its lift $\tilde{g}\in \GSpin(V^{\Phi}_{K_{0}})$ whose similitude satisfies $\ord_{p}(\eta(\tilde{g}))=1$. Then it follows that $\tilde{g}$ swaps $x_{\circ}$ and $x_{\bullet}$ in $\breve{\calM}^{\Diamond}_{\Lambda}(\FF)/p^{\ZZ}$ but then $g$ swaps  $x_{\circ}$ and $x_{\bullet}$ in $\breve{\calM}_{\Lambda}(\FF)$ .  This reasoning along with the fact that $\SO(V^{\Phi}_{K_{0}})$ acts transitively on the set of almost-self-dual lattices in $V^{\Phi}_{K_{0}}$ implies that $\SO(V^{\Phi}_{K_{0}})$  acts transitively on the set $\breve{\calM}_{\ssp}(\FF)$.  

Recall that  giving an orientation on $\Lambda/\Lambda^{\vee}$ is the same as choosing an isotropic line in $\Lambda_{W_{0}}/\Lambda_{W_{0}}^{\vee}$ which in turn is equivalent to give special lattice $L$ such that $\Lambda^{\vee}_{W_{0}}\subset^{1} L\subset ^{1}\Lambda_{W_{0}}$ which is the same as choosing a point from the two points in  $\breve{\calM}_{\Lambda}(\FF)$. Thus the stabilizer of a point in $\breve{\calM}_{\ssp}(\FF)$ in $\SO(V^{\Phi}_{K_{0}})$ is isomorphic to $\SO(\Lambda,\pm)$. This finishes the proof.
\end{proof}

\section{Comparison with Affine Deligne-Lusztig varieties}

In this section we would like indicate how the results proved in Theorem \ref{thmZ} can be extended to cover all the Coxeter type Shimura varieties.  The Coxeter type Shimura varieties are introduced by G\"{o}rtz and He \cite{GH15} in terms of the affine Deligne-Lusztig varieties. In the following we will review their theory and we abbreviate affine Deligne-Lusztig variety by ADLV.
\subsection{Affine Deligne Lusztig variety} Let $F$ be a finite extension of $\QQ_{p}$ and $\breve{F}$ be the completion of the maximal unramified extension of $F$. Let $G$ be a connected reductive group over $F$ and we write $\breve{G}$ its base change to $\breve{F}$. Then $\breve{G}$ is quasi split and we choose a maximal split torus $S$ and denote by $T$ its centralizer.  We know $T$ is a maximal torus and we denote by $N$ its normalizer. The relative Weyl group is defined to be $W_{0}=N(\breve{F})/ T(\breve{F})$. Let $\Gamma$ be the Galois group of $\breve{F}$ and we have the following Kottwitz homomorphism:
$$\kappa_{G}: G(\breve{F})\rightarrow X_{*}(\breve{G})_{\Gamma}.$$
Denote by $\widetilde{W}$ the Iwahori Weyl group of $\breve{G}$ which is by definition $\widetilde{W}=N(\breve{F})/ T(\breve{F})_{1}$ where $T(\breve{F})_{1}$ is the kernel of the Kottwitz homomorphism for $T(\breve{F})$. Inside the Iwahori Weyl group $\widetilde{W}$, there is a copy of the affine Weyl group $W_{a}$ which can be identified with $N(\breve{F})\cap G(\breve{F})_{1}/ T(\breve{F})_{1}$ where $G(\breve{F})_{1}$ is the kernel of the Kottwitz morphism for $G$. The group $\tilde{W}$ is not quite a Coxeter group while $W_{a}$ is generated by the affine reflections denoted by $\tilde{\mathbb{S}}$ and $(\widetilde{W}, \tilde{\mathbb{S}})$ form a Coxeter system. We in fact have $\widetilde{W}= W_{a}\rtimes \Omega$ where $\Omega$ is the normalizer of a fixed base alcove and more canonically $\Omega=X_{*}(T)_{\Gamma}/ X_{*}(T_{sc})_{\Gamma}$ where $T_{sc}$ is the preimage of $T\cap G_{der}$ in the simply connected cover $G_{sc}$.

Let $\mu\in X_{*}(T)$ be a minuscule cocharacter of $G$ over $\breve{F}$ and $\lambda$ its image in $X_{*}(T)_{\Gamma}$.  We denote by $\tau$ the projection of $\lambda$ in $\Omega$. The \emph{admissible subset} of $\widetilde{W}$ is defined to be
$$\Adm(\mu)=\{w\in \widetilde{W}; w \leq x(\lambda) \text{ for some }x\in W_{0} \}.$$
Here $\lambda$ is considered as a translation element in $\widetilde{W}$. Let $K\subset\tilde{\mathbb{S}}$ and $\breve{K}$ its corresponding parahoric subgroup. Let $\widetilde{W}_{K}$ be the subgroup defined by $N(\breve{F})\cap \breve{K}/T(\breve{F})_{1}$. We have the following decomposition $\breve{K}\backslash G(\breve{F})/\breve{K}= \widetilde{W}_{K}\backslash \widetilde{W}/ \widetilde{W}_{K}$. Therefore we can define a relative position map 
$$\inv: G(\breve{F})/\breve{K}\times G(\breve{F})/\breve{K}\rightarrow \widetilde{W}_{K}\backslash \widetilde{W}/ \widetilde{W}_{K}.$$

For $w\in \widetilde{W}_{K}\backslash \widetilde{W}/ \widetilde{W}_{K}$ and $b\in G(\breve{F})$, we define the \emph{affine Deligne-Lusztig variety}
to be the set 
$$X_{w}(b)=\{g\in G(\breve{F})/\breve{K}; \inv(g, b\sigma(g))=w\}.$$
Thanks to the work of \cite{BS-Inv17} and \cite{Zhu-Ann17}, this set can be viewed as an ind-closed-subscheme in the affine flag variety $\breve{G}/\breve{K}$. In this note, we only consider it as a set. The Rapoport-Zink space is not directly related to the affine Deligne-Lusztig variety but rather to the following union of affine Deligne-Lusztig varieties
$$X(\mu, b)_{K}=\{g\in G(\breve{F})/ \breve{K}; g^{-1}b\sigma(g)\in \breve{K}w\breve{K}, w\in \Adm(\mu) \}.$$
We recall the group $J_{b}$ is defined by the $\sigma$-centralizer of $b$ that is $$J_{b}(R)=\{g\in G(R\otimes_{F}\breve{F}); g^{-1}b\sigma(g)=b\}$$ for any $F$-algebra $R$. In the following we will assume that $b$ is basic and in this case $J_{b}$ is an inner form of $G$.

\subsection{Coxeter type ADLV}We define $\Adm^{K}(\mu)$ to be the image of  $\Adm(\mu)$ in $\widetilde{W}_{K}\backslash\widetilde{W}/\widetilde{W}_{K}$ and $^{K}\widetilde{W}$ to be the set of elements of minimal length in $\widetilde{W}_{K}\backslash\widetilde{W}$.  We define the set $\mathrm{EO}^{K}(\mu)=\Adm^{K}(\mu)\cap ^{K}\widetilde{W}$. For $w\in W_{a}$, we set 
$$\mathrm{supp}_{\sigma}(w\tau)=\bigcup_{n\in \ZZ}(\tau\sigma)^{n}(\mathrm{supp}(w)).$$
If the length $l(w)$ of $w$ agrees with the cardinality of $\mathrm{supp}_{\sigma}(w\tau)/\langle\tau\sigma\rangle$, we say $w\tau$ is a $\sigma$-Coxeter element. We denote by $\mathrm{EO}^{K}_{\sigma,\mathrm{cox}}(\mu)$ the subset of $\mathrm{EO}^{K}(\mu)$ such that $w$ is a $\sigma$-Coxeter element and $\mathrm{supp}_{\sigma}(w)$ is not $\widetilde{\mathbb{S}}$. A \emph{$K$-stable piece} is a subset of $G(\breve{F})$ of the form $\breve{K}\cdot_{\sigma}\breve{I}w\breve{I}$ where $\cdot_{\sigma}$ means $\sigma$-conjugation and $\breve{I}$ is an Iwahori subgroup and $w\in {^{K}\widetilde{W}}$. Then we define the Ekedahl-Oort stratum attached to $w\in \mathrm{EO}^{K}(\mu)$ of $X(\mu, b)_{K}$ by the set $$X_{K,w}(b)=\{g\in G(\breve{F})/\breve{K}; g^{-1}b\sigma(g)\in  \breve{K}\cdot_{\sigma}\breve{I}w\breve{I}\}.$$  Then by \cite{GH15} we have the following EO-stratification 
\begin{equation}X(\mu, b)_{K}=\bigcup_{w\in\mathrm{EO}^{K}(\mu)}X_{K,w}(b).\end{equation} 
The case when 
\begin{equation}\label{ADLV-EO}X(\mu, b)_{K}=\bigcup_{w\in\mathrm{EO}^{K}_{\sigma,\mathrm{cox}}(\mu)}X_{K,w}(b)\end{equation}
is particular interesting and when this happens we say the datum $(G, \mu, K)$ is of Coxeter type. The datum $(G, \mu, K)$ being Coxeter type or not depends only on the associated datum $(\widetilde{W}, \lambda, K, \sigma)$ where $\lambda$ is the image of $\mu\in X_{*}(T)_{\Gamma}$ and $\sigma$ is the induced automorphism of the Frobenius $\sigma$ on the local Dynkin diagram. The set of $(G, \mu, K)$ is classified in \cite[Theorem 5.11]{GH15}. This includes the orthogonal case we studied in the previous sections.
\begin{itemize}
\item The odd orthogonal case  corresponds to $G=\SO_{n}$ and the datum $$(\widetilde{B}_{m}, \omega^{\vee}_{1}, \mathbb{S}, id )$$ with $n=2m+1$. The affine Dynkin diagram is given by
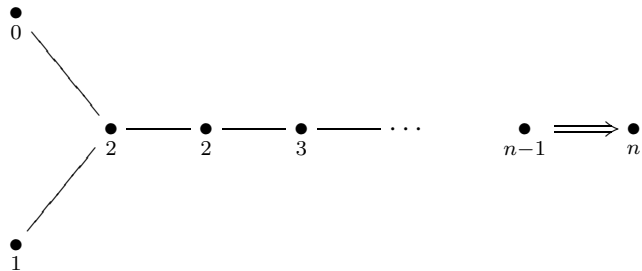
\begin{displaymath}
  \xymatrix{ &\underset{0}\bullet \ar@{-}[dr]& \\
                  &&\underset{2}\bullet \ar@{-}[r] &\underset{2}\bullet &\underset{3}\bullet\ar@{-}[l] &\cdots\ar@{-}[l]&\underset{n-1}\bullet\ar@{=>}[r]&\underset{n}\bullet\\
                  &\underset{1}\bullet\ar@{-}[ur]}\\
\end{displaymath}
with Frobenius acts trivially on the diagram.

\item The even orthogonal case with $\det\neq (-1)^{m}$ corresponds to $G=\SO_{n}$ and the datum $$(\widetilde{D}_{m}, \omega^{\vee}_{1}, {\mathbb{S}}, id)$$
with $n=2m$.
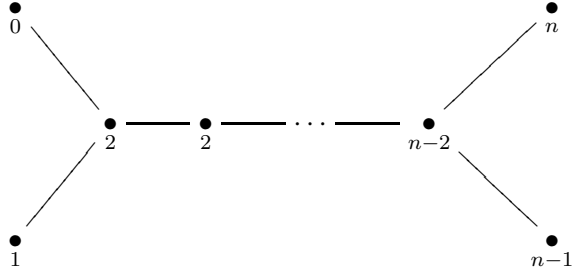
\begin{displaymath}
  \xymatrix{ &\underset{0}\bullet \ar@{-}[dr]& & & & &\underset{n}\bullet \ar@{-}[dl]\\
                  &&\underset{2}\bullet \ar@{-}[r] &\underset{2}\bullet &\cdots\ar@{-}[l] &\underset{n-2}\bullet\ar@{-}[l]\\
                  &\underset{1}\bullet\ar@{-}[ur] & & & & &\underset{n-1}\bullet \ar@{-}[ul]\\}
\end{displaymath}
with Frobenius acts trivially on the diagram.

\item The even orthogonal case with $\det= (-1)^{m}$ corresponds to $G=\SO_{n}$ and the datum $$(\widetilde{D}_{m}, \omega^{\vee}_{1}, {\mathbb{S}}, \sigma)$$
with $n=2m$.
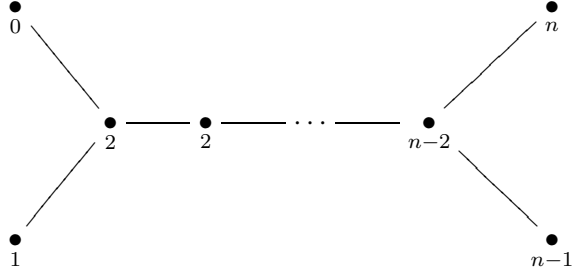
\begin{displaymath}
  \xymatrix{ &\underset{0}\bullet \ar@{-}[dr]& & & & &\underset{n}\bullet \ar@{-}[dl]\\
                  &&\underset{2}\bullet \ar@{-}[r] &\underset{2}\bullet &\cdots\ar@{-}[l] &\underset{n-2}\bullet\ar@{-}[l]\\
                  &\underset{1}\bullet\ar@{-}[ur] & & & & &\underset{n-1}\bullet \ar@{-}[ul]\\}
\end{displaymath}
with Frobenius fixes the nodes $0,1$ and switches the nodes $n-1, n$.

\end{itemize}

\subsection{Bruhat-Tits stratification of ADLV} Now we assume that $K$ is a maximal proper subset of $\tilde{\mathbb{S}}$ such that $\sigma(K)=K$. Consider the following set
$$\mathcal{J}=\{\Sigma\subset\tilde{\mathbb{S}}; \emptyset\neq \Sigma\text{ is }  \tau\sigma\text{-stable}\text{ and }d(v)=d(v^{\prime})\text{ for every } v,v^{\prime}\in \Sigma\}.$$ where $d(v)$ is the distance between $v$ and the unique vertex not in $K$.
In fact every $w\in \mathrm{EO}^{K}_{\sigma,\mathrm{cox}}(\mu)$ corresponds to a $\Sigma\in \mathcal{J}$ and we write $w$ as $w_{\Sigma}$.  If $(G,\mu, K)$ is of Coxeter type, for any $w_{\Sigma}\in\mathrm{EO}^{K}_{\sigma,\mathrm{cox}}(\mu)$, \begin{equation}\label{EO-DL}X_{K,w_{\Sigma}}(b)=\bigcup_{i\in J_{b}/J_{b}\cap \breve{K}_{\tilde{\mathbb{S}}-\Sigma}}i. X(w_{\Sigma}).\end{equation} Here $\breve{K}_{\tilde{\mathbb{S}}-\Sigma}$ is the parahoric subgroup associated to the set $\tilde{\mathbb{S}}-\Sigma$ and $X(w_{\Sigma})$ is a classical Deligne-Lusztig variety defined by 
\begin{equation}\label{DL-sigma}
X(w_{\Sigma})=\{g\in \breve{K}_{\mathrm{supp}_{\sigma}(w_{\Sigma})}/\breve{I}; g^{-1}\tau\sigma(g)\in \breve{I}w_{\Sigma}\breve{I}\}
\end{equation} 
which is a Deligne-Lusztig variety attached to the maximal reductive quotient $G_{w}$ of the special fiber of $\breve{K}_{\tilde{\mathbb{S}}-\Sigma}$. Combine \eqref{ADLV-EO} and \eqref{EO-DL} we arrive at the following \emph{Bruhat-Tits stratification} of $X(\mu, b)_{K}$:
\begin{equation}\label{EO-dec}
X(\mu, b)_{K}=\bigcup_{J_{b}/J_{b}\cap \ker(\kappa_{G})}\bigcup_{w_{\Sigma}\in  \mathrm{EO}^{K}_{\sigma,\mathrm{cox}}} \mathcal{X}^{\circ}_{\Sigma}
\end{equation}
where 
\begin{equation}\label{MSigma}
\mathcal{X}^{\circ}_{\Sigma}=\bigcup_{i\in J_{b}\cap\ker(\kappa_{G})/J_{b}\cap \breve{K}_{\tilde{\mathbb{S}}-\Sigma}}i. X(w_{\Sigma}).
\end{equation}
Here the index set is related to the Bruhat-Tits building of $J_{b}$ in the following way. The group  $J_{b}\cap\ker(\kappa_{G})$ acts on the set of faces of type $\Sigma$ transitively and  $J_{b}\cap \breve{K}_{\tilde{\mathbb{S}}-\Sigma}$ is precisely the stabilizer of the face of type $\Sigma$ in the base alcove.

\subsection{Computations in the orthogonal case}Now we let $G=\SO(V)$ as in the previous section.

\subsubsection{Odd orthogonal case} In this case, we have
\begin{center}
\begin{tabular}{lllll}
$\Sigma$                                     & \{0, 1\}                    & \{2\}                     & \{$i$+1\}                              \\
$w_{\Sigma}$                               & $\tau $     & $s_{0}\tau$ & $s_{0}s_{2}\cdots s_{i}\tau$ &         \\
$\tilde{\mathbb{S}}-\Sigma$             & $\{2,\cdots, m\}$                 & $\{0,1\}\cup \{3,\cdots, m\}$                      &$ \{0,1,\cdots, i\}\cup\{i+2,\cdots, m\}                           $\\
$\mathrm{supp}_{\sigma}(w_{\Sigma})$ & $\emptyset $& \{0,1\}                       & $\{0,1,2\cdots i\}   $                      .                            
\end{tabular}
\end{center}
Here $i$ lies in the range $[2, m-1]$.

\subsubsection{Even orthogonal case with $\det= (-1)^{m}$} In this case, we have
\begin{center}
\begin{tabular}{lllll}
$\Sigma$                                     & \{0, 1\}                    & \{2\}                     & \{$i$+1\}                              \\
$w_{\Sigma}$                               & $\tau $     & $s_{0}\tau$ & $s_{0}s_{2}\cdots s_{i}\tau$ &         \\
$\tilde{\mathbb{S}}-\Sigma$             & $\{2,\cdots, m\}$                 & $\{0,1\}\cup \{3,\cdots, m\}$                      &$ \{0,1,\cdots, i\}\cup\{i+2,\cdots, m\}                           $\\
$\mathrm{supp}_{\sigma}(w_{\Sigma})$ & $\emptyset $& \{0,1\}                       & $\{0,1,2\cdots i\}   $                      .                            
\end{tabular}
\end{center}

Here $i$ lies in the range $[2, m-2]$.

\subsubsection{Even orthogonal case with $\det\neq (-1)^{m}$} In this case, we have
\begin{center}
\begin{tabular}{ c c c c c}
$\Sigma$                                     & \{0, 1\}                    & \{2\}                     & \{$i$+1\}               \\
$w_{\Sigma}$                               & $\tau $     & $s_{0}\tau$ & $s_{0}s_{2}\cdots s_{i}\tau$  \\
$\tilde{\mathbb{S}}-\Sigma$             & $\{2,\cdots, m\}$                 & $\{0,1\}\cup \{3,\cdots, m\}$                      &$ \{0,1,\cdots, i\}\cup\{i+2,\cdots, m\}$   \\
$\mathrm{supp}_{\sigma}(w_{\Sigma})$ & $\emptyset $& \{0,1\}                       & $\{0,1,2\cdots i\}   $            
\end{tabular}
\end{center}

\begin{center}
\begin{tabular}{lllll}
$\Sigma$                                     & \{$m-1$\}                    & \{$m$\}                                    \\
$w_{\Sigma}$                               & $s_{0}s_{2}\cdots s_{m-2}s_{m}\tau $     & $s_{0}s_{2}\cdots s_{m-2}s_{m-1}\tau$          \\
$\tilde{\mathbb{S}}-\Sigma$             & $\{0,\cdots, m-2\}\cup\{m\}$                 & $\{0,\cdots, m\}$                \\
$\mathrm{supp}_{\sigma}(w_{\Sigma})$ & $\{0,\cdots, m-2\}\cup\{m\}$                     & $\{0,\cdots, m-1\}$              .                            
\end{tabular}
\end{center}

Here $i$ lies in the range $[2, m-2]$.

Each label $\Sigma$ for the EO-strata $\mathcal{X}^{\circ}_{\Sigma}$ determines a type for the vertex lattices. We write $t_{\Sigma}$ the type determined by $\Sigma$. This relation is given by the following formula
\begin{equation}
t_{\Sigma}=2(l(w_{\Sigma})+1). 
\end{equation} 

\begin{lemma}\label{RZ-ADL}
There is a bijection between $\breve{\mathcal{M}}(\FF)$ and the set $X(\mu, b)_{K}$.
\end{lemma}
\begin{proof}
This follows from \cite[Theorem 2.4.10]{HP17} and \cite[Proposition 2.4.3]{HP17}.
\end{proof}

Under the bijection in Lemma \ref{RZ-ADL}, the EO-stratum  $\mathcal{X}^{\circ}_{\Sigma}$ correspond to the union of Bruhat-Tits strata of a fixed type:
$$\mathcal{X}^{\circ}_{\Sigma}=\bigcup_{\Lambda, t_{\Lambda}= t_{\Sigma}} \BT_{\Lambda}.$$
We define the superspecial locus of $X(\mu, b)_{K}$ to be the \emph{minimal EO stratum} that is the stratum whose corresponding element $w_{\Sigma}$ is of minimal length. In all the cases computed above the minimal EO stratum corresponds to $\Sigma=\{0,1\}$ and 
$$\mathcal{X}^{\circ}_{\{0,1\}}=\bigcup_{\Lambda, t_{\Lambda}= 2} \BT_{\Lambda}$$
is exactly the union of type $2$ vertex strata which is $\breve{\mathcal{M}}_{\ssp}(\FF)$. Then we can prove Theorem \ref{thmZ} can be proved using \eqref{MSigma}
$$\mathcal{X}^{\circ}_{\{0,1\}}=\bigcup_{i\in J_{b}/J_{b}\cap \breve{K}_{\tilde{\mathbb{S}}-\{0,1\}}}i. X(w_{\{0,1\}}).$$
 Indeed since $\mathrm{supp}_{\sigma}(w_{\{0,1\}})=\emptyset$, one deduces that $X(w_{\{0,1\}})$ is a single point from \eqref{DL-sigma} and therefore $\mathcal{X}^{\circ}_{\{0,1\}}=\breve{\mathcal{M}}_{\ssp}(\FF)$ is given by the homogeneous space $ J_{b}/J_{b}\cap \breve{K}_{\tilde{\mathbb{S}}-\{0,1\}}$ which is just $Z$ as in Theorem \ref{thmZ}.

We finish this section by remarking that one can extend the above computation to any Coxeter type ADLV and obtain a uniformization of the minimal Bruhat-Tits stratum by a homogeneous space like $Z$. Indeed, one can go through all the cases in \cite[6.3]{GH15} and check that all the minimal Bruhat-Tits strata are of dimension $0$ and are unique. 

\section{Liftings the superspecial locus}
In this section we describe the lifts of the superspecial points in the integral points of the Rapoport-Zink space. For this we need to recall the global construction of the Rapoport-Zink space in \cite{HP17}. In this section we abuse some notations. We will  denote by $(V, Q)$ a quadratic space $\ZZ_{(p)}$ of signature $(n-2, 2)$ with $n\geq 3$. Let $G^{\Diamond}$ be the group $\Gspin(V)$ over $\ZZ_{(p)}$ and $G$ be the group $\SO(V)$ over $\ZZ_{(p)}$.  Let $K^{\Diamond}\subset G^{\Diamond}(\mathbb{A}_{f})$ be an open compact subset such that $K^{\Diamond}=K^{\Diamond}_{p}K^{\Diamond p}$ with $K^{\Diamond}_{p}=G^{\Diamond}(\ZZ_{p})$ and $K^{\Diamond p}$ sufficiently small. Similarly $K=K^{p}K_{p}$ with $K_{p}=G(\ZZ_{p})$ and $K^{p}$ sufficiently small. 
\subsection{Integral model of $\Gspin$-type and orthogonal type Shimura varieties} We have the integral canonical model $\mathscr{S}^{\Diamond}_{K^{\Diamond}}$ over $W_{0}$ for the Shimura variety for the spin similitude group $G^{\Diamond}$ and the integral canonical model $\mathscr{S}_{K}$ over $W_{0}$ for Shimura variety for the special orthogonal group $G$ constructed by Madapusi Pera \cite{MP16}. There is a closed immersion  $\mathscr{S}^{\Diamond}_{K^{\Diamond}}\hookrightarrow \mathscr{A}_{g}$ where $\mathscr{A}_{g}$ is a suitable Siegel Shimura variety coming from the Hodge embedding. Let $(A^{\KS}, \lambda^{\KS})$ be the \emph{Kuga-Satake} abelian scheme over $\mathscr{S}^{\Diamond}_{K^{\Diamond}}$ which is the pullback of the universal abelian scheme over $\mathscr{A}_{g}$ with its polarization. Consider the special fiber $\mathscr{S}^{\Diamond}_{K^{\Diamond}, \FF}$ of $\mathscr{S}^{\Diamond}_{K^{\Diamond}}$ at $p$ equipped with $(A^{\KS}_{\FF}, \lambda^{\KS}_{\FF})$. Let $\mathbf{H}_{\cris}$ be the first crystalline homology $\mathrm{H}^{\cris}_{1}(A^{\KS}_{\FF})$ of $A^{\KS}_{\FF}$ over $\mathscr{S}^{\Diamond}_{K^{\Diamond}, \FF}$. This is a crystal over $(\mathscr{S}^{\Diamond}_{K^{\Diamond},\FF}/W_{0})_{cris}$, the big crystalline site of $\mathscr{S}^{\Diamond}_{K^{\Diamond}, \FF}$ over $W_{0}$.   We introduce the following notation:
$$\mathbf{H}^{(r,s)}_{\cris}= \mathbf{H}_{\cris}^{\otimes r}\otimes \mathbf{H}_{\cris}^{* \otimes s} $$
where $\mathbf{H}^{*}_{\cris}$ is the $W_{0}$ dual of $\mathbf{H}_{\cris}$. Let $y\in \mathscr{S}^{\Diamond}_{K^{\Diamond},\FF}(\FF)$ be a point, then we write $\mathbf{H}_{\cris, y}$ the specialization of $\mathbf{H}_{\cris}$ to $y$. For each $y$, one can construct a quadratic space 
\begin{equation}\label{special-p}
\mathbf{V}_{\cris,y}\subset \mathbf{H}^{(1,1)}_{\cris,y}
\end{equation}
over $W_{0}$ which is isomorphic to $V_{W_{0}}$ by \cite[Proposition 4.7]{MP16}. Similarly, for each prime $l\neq p$, we denote by $\mathbf{H}_{l}$ be the relative $l$-adic Tate module of $A^{\KS}_{\FF}$ and denote by $\mathbf{H}_{l,y}$ its specialization to $y$. One can again construct a quadratic space 
\begin{equation}\label{special-l}
\mathbf{V}_{l, y}\subset\mathbf{H}^{(1,1)}_{l, y} 
\end{equation}
as in \cite[3.12]{MP16}. We will need the following characterization of $\mathbf{V}_{\cris, y}$. 

\begin{proposition}\label{Vcris}
Set $\mathbf{H}_{\dR, y}=\mathbf{H}_{\cris, y}\otimes \FF$ and $\mathbf{V}_{\dR, y}=\mathbf{V}_{\cris, y}\otimes \FF$. Then there is a canonical isotropic line ${F}^{1}\mathbf{V}_{\dR, y}\subset \mathbf{V}_{\dR, y}$ such that the Hodge filtration $\mathrm{Fil}^{1}\mathbf{H}_{\dR, y}$ is cut out by ${F}^{1}\mathbf{V}_{\dR, y}$. That is ${F}^{1}\mathbf{H}_{\dR, y}=\Ker({F}^{1}\mathbf{V}_{\dR, y})$. 
\end{proposition}

\begin{proof}
This is proved in \cite[Proposition 4.7]{MP16}.
\end{proof}
Since $\mathscr{S}^{\Diamond}_{K^{\Diamond},\FF}\rightarrow \mathscr{S}_{K,\FF}$ is an \'{e}tale covering, the crystals $\mathbf{V}_{\cris,x}$ and $\mathbf{H}_{\cris,x}$ constructed above naturally descends to crystals on  $(\mathscr{S}_{K,\FF}/W_{0})_{\cris}$. We use the same notation to denote them.

Let $\breve{\calM}_{\GSp(C)}$ be the Siegel Rapoport-Zink space and let $\Theta_{\GSp}:\breve{\calM}_{\GSp(C)}\rightarrow  \widehat{\mathscr{A}}_{g}$ be the Rapoport-Zink uniformization map \cite[Theorem 6.21]{RZ96}. 
By construction \cite[Definition 3.2.6]{HP17} the Rapoport-Zink space $\breve{\calM^{\Diamond}}$ is realized as an open and closed formal subscheme of $$\breve{\calM}_{\GSp(C)}\times_{\widehat{\mathscr{A}}_{g}} \widehat{\mathscr{S}}^{\Diamond}_{K^{\Diamond}, W_{0}}$$ and therefore naturally comes equipped with a map $$\Theta: \breve{\calM^{\Diamond}} \rightarrow \widehat{\mathscr{S}}^{\Diamond}_{K^{\Diamond}, W_{0}}$$
where $\widehat{\mathscr{S}}^{\Diamond}_{K^{\Diamond}, W_{0}}$ is the completion of $\mathscr{S}^{\Diamond}_{K^{\Diamond}, W_{0}}$ along its special fiber at $p$. Let $x\in \breve{\calM^{\Diamond}}(\FF)$ and $y=\Theta(x)\in  \mathscr{S}^{\Diamond}_{K^{\Diamond},\FF}(\FF)$. Then the construction of $\Theta$ in fact furnishes an isomorphism $$ \breve{\calM}^{\Diamond}_{x}\xrightarrow{\cong}  \widehat{\mathscr{S}}^{\Diamond}_{K^{\Diamond}, y}$$ where  $\breve{\calM}^{\Diamond}_{x}$ is the completed local ring of $\breve{\calM}^{\Diamond}$ at $x$ and $\widehat{\mathscr{S}}^{\Diamond}_{K^{\Diamond}, y}$ is the completed local ring of $\widehat{\mathscr{S}}^{\Diamond}_{K^{\Diamond}}$ at $y$. All the above constructions carry over to the setting of orthogonal type Shimura varieties and Rapoport-Zink spaces. In particular, we have a map
$$\Theta: \breve{\calM} \rightarrow \widehat{\mathscr{S}}_{K, W_{0}}$$ and an identification $\breve{\calM}_{x}\xrightarrow{\cong}  \widehat{\mathscr{S}}_{K, y}$ for  $x\in \breve{\calM}(\FF)$ and $y=\Theta(x)\in  \mathscr{S}_{K,\FF}(\FF)$. 

In particular, the above identification allows us to transfer the results in Proposition \ref{Vcris} to the setting of $p$-divisible groups. 
\begin{lemma}
There is an identification $\mathbf{H}_{\cris,y}=\mathbb{D}(X_{x})(W_{0})$ with $y=\Theta(x)$ and in this case $\mathbf{V}_{cris, y}$ is precisely the special lattice $\Phi_{*}(L_{x})$ associated to $x\in \breve{\mathcal{M}}(\FF)$. 
\end{lemma}

\begin{proof}
The identification $\mathbf{H}_{\cris, y}=\mathbb{D}(X_{x})(W_{0})$ is clear. To prove the second assertion, notice that by construction $\mathbf{V}_{\cris, y}$ is contained in $L^{\sharp}_{x}  =\{z\in V_{K_{0}}: zM_{y}\subset M_{y}\}$. Then the claim in the lemma follows from the fact that both $\mathbf{V}_{\cris, y}$ and $L^{\sharp}_{x}$ are self-dual lattices in $V_{K_{0}}$.
\end{proof}

\begin{lemma}\label{F1L}
There is a canonical isotropic line ${F}^{1}\Phi_{*}(L_{x,\FF})$ in $\Phi_{*}(L_{x,\FF})$ which is characterized as the orthogonal complement of $\Phi_{*}(L_{x, \FF})\cap L_{x,\FF}$ in $\Phi_{*}(L_{x, \FF})$.
\end{lemma}

\begin{proof}
The existence of $F^{1}\Phi_{*}(L_{x, \FF})$ follows from Proposition \ref{Vcris}. Recall that by \eqref{special-lattice}, $L_{x, \FF} =\{z\in V_{\FF}: zM_{1, y, \FF}\subset M_{1,y,\FF}\}$ and $\Phi_{*}(L_{x, \FF}) =\{z\in V_{\FF}: zM_{y, \FF}\subset M_{y, \FF}\}.$ By the definition of special lattice, we have $L_{x,\FF}\cap \Phi_{*}(L_{x, \FF})\subset^{1} \Phi_{*}(L_{x, \FF})$. Since $F^{1}\Phi_{*}(L_{x, \FF})=\{z\in \Phi_{*}(L_{x,\FF}): zM_{1,x, \FF}=0\}$, $F^{1}\Phi_{*}(L_{x, \FF})\subset (L_{x,\FF}\cap \Phi_{*}(L_{x, \FF}))^{\perp}$. The inclusion is an equality since both sides of the inclusion is of dimension $1$. 
\end{proof}

\subsection{Deformation theory}
 
Let $x\in \breve{\calM}(\FF)$ be a closed point. Then the completed local ring  $\breve{\calM}_{x}$ represents the set valued functor that classifies the lifts of $x$.  In particular we have the following identification  
$$\breve{\mathcal{M}}_{x}(W_{0})=\{\text{lifts of }x \text{ in }\breve{\calM}(W_{0})\}.$$

Combing the bijection $\breve{\calM}_{x}\xrightarrow{\cong}  \widehat{\mathscr{S}}_{K, y}$ and Proposition \ref{Vcris}, we obtain the following result.

\begin{theorem}\label{lift}
Let $x\in \breve{\mathcal{M}}(\FF)$. There is an identification 
$$\breve{\mathcal{M}}_{x}(W_{0})=\{\text{isotropic line } F^{1}\Phi_{*}(L_{x})\subset \Phi_{*}(L_{x})\text{ lifting } F^{1}\Phi_{*}(L_{ x, \FF}) \subset \Phi_{*}(L_{x, \FF})\}.$$
\end{theorem}

\begin{proof}
This follows from Grothendieck-Messing theory that the lifts of $y=\Theta(x)$ is the same as the lifts of the Hodge filtration $F^{1}\mathbf{H}_{\dR, y}\subset\mathbf{H}_{\dR, y}$ in $\mathbf{H}_{\cris, y}$. Since $F^{1}\mathbf{H}_{\dR, y}= \Ker(F^{1}\mathbf{V}_{\dR, y})$, the result follows from the identification $\mathbf{H}_{\cris,x}=\mathbb{D}(X_{y})(W_{0})$. See also \cite[Theorem 4.1.7]{LZ18}.
\end{proof}

\begin{corollary}
Let $x\in \breve{\mathcal{M}}(\FF_{p^{2}})$. There is an identification 
\begin{equation*}
\begin{split}
&\breve{\mathcal{M}}_{x}(\calO_{K})=\{\text{isotropic line } F^{1}\Phi_{*}(L_{x})\subset \Phi_{*}(L_{x})\text{ lifting } F^{1}\Phi_{*}(L_{ x, \FF}) \subset \Phi_{*}(L_{x, \FF})\\
&F^{1}\Phi_{*}(L_{x})\subset \Phi_{*}(L_{x})\text{ is defined over } \calO_{K}\}.\\
\end{split}
\end{equation*}
\end{corollary}

Suppose now that $\Lambda$ is a type $2$ vertex lattice. Recall that the corresponding Bruhat-Tits stratum has only two reduced points $\{x_{\circ}, x_{\bullet}\}$ defined over $\FF_{p^{2}}$. We consider the deformations of these points together. Define $$Def_{\Lambda}(\mathcal{O}_{K})=\breve{\mathcal{M}}_{x_{\circ}}(\mathcal{O}_{K})\bigsqcup \breve{\mathcal{M}}_{x_{\bullet}}(\mathcal{O}_{K}).$$ 

First we assume that $\tilde{x}\in Def_{\Lambda}(\mathcal{O}_{K})$ is a lift of $x_{\circ}$. This is equivalent to giving an isotropic line $F^{1}L\subset L_{x_{\circ}}$. Then we consider the plane $W_{\mathcal{O}_{K}}=F^{1}L+\Phi_{*}(F^{1}L)$ over $\mathcal{O}_{K}$. As $L_{x_{\circ}}+\Phi_{*}(L_{x_{\circ}})$ is invariant under $\Phi$ and $L_{x_{\circ}}$ is defined over $\mathcal{O}_{K}$, $W_{\mathcal{O}_{K}}$ descends to  an oriented plane $W$ over $\ZZ_{p}$ with orientation given by the line $F^{1}L$.  We have also seen that $L_{x_{\circ}}+\Phi_{*}(L_{x_{\circ}})=\Lambda_{\mathcal{O}_{K}}$. Therefore $W$ is an oriented two plane contained in $\Lambda$. If $\tilde{x}$ is a lift of $x_{\bullet}$, then this is equivalent to giving an isotropic line $F^{1}L\subset L_{x_{\bullet}}$. But $L_{x_{\bullet}}=\Phi_{*}(L_{x_{\circ}})$ and therefore this is equivalent to giving an isotropic line in $L_{x_{\circ}}$. Combining the above with a little extra argument, we have the following lemma.

\begin{lemma}\label{Keylemma}
There is a bijection between $Def_{\Lambda}(\mathcal{O}_{K})$ and $\{\text{oriented plane } W \subset \Lambda\}$.
\end{lemma}
\begin{proof}
We give a map from $Def_{\Lambda}(\mathcal{O}_{K})$ to $\{\text{oriented plane } W \subset \Lambda \}$. By the previous paragraph, we set $L=L_{x_{\circ}}$ and $F^{1}L$ to be an isotropic line in $L_{x_{\circ}}$. The oriented plane is already defined by $W$ such that $W_{\mathcal{O}_{K}}=F^{1}L+\Phi_{*}(F^{1}L)$ in the previous paragraph. Note that the orientation of $W$ is given by $F^{1}L$. Conversely given an oriented plane $W\subset \Lambda$, we choose a vector $e\in W_{\mathcal{O}_{K}}$ that generates the isotropic line which determines the orientation of $W$. We set $L=\mathcal{O}_{K}.e+\Lambda^{\vee}_{\mathcal{O}_{K}}$. Note that $e$ is not stable under $\Phi$ and therefore $L$ is a special lattice with $\Lambda^{\vee}_{\mathcal{O}_{K}}\subset L\subset \Lambda_{\mathcal{O}_{K}}$. The $\mathcal{O}_{K}$-line $\mathcal{O}_{K}.e$ gives an isotropic line in $L$ and thus $\mathcal{O}_{K}$-line $\mathcal{O}_{K}.\Phi(e)$ gives an isotropic line in $\Phi_{*}(L)$. This shows the claimed bijection.


\end{proof}

Finally consider all the lifts of the superspecial points $$Def_{\ssp}(\calO_{K})=\bigcup_{\Lambda, t_{\Lambda=2}}Def_{\Lambda}(\calO_{K}).$$ The following theorem confirms the second part of Gross's conjecture.

\begin{theorem}\label{Keythm}
There is a bijection between $Def_{\ssp}(\mathcal{O}_{K})$ and $$\{\text{an oriented two plane } W \subset V^{\Phi}_{K_{0}} \text{ and a lattice } M\subset W^{\perp}  \}.$$
\end{theorem}
\begin{proof}
Given a pair $(W, M)$, consider the sum $\Lambda=W+M$, this gives an almost-self-dual lattice. Then we have a bijection between those $(W, M)$ with $\Lambda=W+M$ and $Def_{\Lambda}(\mathcal{O}_{K})$ via the bijection in Lemma \ref{Keylemma}.  The theorem immediately follows by taking the union over all the vertex lattices  of type $2$. 
\end{proof}

\begin{corollary}\label{corY}
The lift of the superspecial locus $Def_{\ssp}(\mathcal{O}_{K})$ is uniformized by the homogeneous space $Y= \SO(V^{\Phi}_{K_{0}})/\SO(W,\iota)\times \SO(M)$ where $W\subset V^{\Phi}_{K_{0}}$ is a fixed oriented plane and $M$ a fixed lattice in $W^{\perp}$.
\end{corollary}
\begin{proof}
This follows from the fact that $\SO(V^{\Phi}_{K_{0}})$ acts transitively on the set of pairs $(W, M)$ and Theorem \ref{Keythm}
\end{proof}

\subsection{Period morphism}
Let $\mathcal{Q}\subset \mathbb{P}(V)$ be the orthogonal flag variety. It is by definition the classifying space of isotropic lines in $V$. We will consider its associated rigid analytic variety. We have the period morphism:

\begin{equation}
\pi_{\dR}: \calM^{\mathrm{rig}}\rightarrow \mathcal{Q}. 
\end{equation}

The image of this map $\mathcal{Q}^{a}$ is called  the \emph{admissible locus} of the orthogonal flag variety $\mathcal{Q}$. One also has the so called \emph{weakly admissible locus} $\mathcal{Q}^{wa}(K_{0})$ which can be described as the set of isotropic lines in $V^{\Phi}_{K_{0}}\otimes K_{0}$ that is not contained in any rational isotropic space contained in $V^{\Phi}_{K_{0}}$. In fact we have the equality $\mathcal{Q}^{a}=\mathcal{Q}^{wa}$ which follows from that the local Shimura datum being fully Hodge-Newton decomposable \cite{CFS18} and \cite{Shen-gen}. The period map $\pi_{\dR}$ is defined for a general Hodge-type Rapoport-Zink spaces is defined by \cite{Kim18} and extended to the abelian case in \cite{Shen-gen}. Let $X^{\Diamond}$ be the universal $p$-divisible group over $\breve{\calM}^{\Diamond}$ with crystalline tensors $(t_{\alpha})$ then the period morphism for the $\GSpin$-type Rapoport-Zink space is given by $(\mathrm{Fil}^{1} \mathbb{D}(X^{\Diamond}), t_{\alpha})^{rig}$. We need a more explicit version of this map in the following situation: let $x_{0}\in \breve{\mathcal{M}}(\FF)$, we will describe the restrction of $\pi_{\dR}$ to $\breve{\mathcal{M}}^{\rig}_{x_{0}}(K_{0})=\breve{\mathcal{M}}_{x_{0}}(W_{0})$ where the equality follows from partial properness of $\breve{\mathcal{M}}$. Now given a lift $$x=(F^{1}\Phi_{*}(L)\subset \Phi_{*}(L_{x_{0}}))\in \breve{\mathcal{M}}^{\rig}_{x_{0}}(K_{0}),$$ it is easy to see that $\pi_{\dR}(x)$ is simply given by $F^{1}\Phi_{*}(L)\in \mathcal{Q}^{a}(K_{0})$. By the definition of the admissible locus $\mathcal{Q}^{a}$, the rational structure of the $V\otimes K_{0}$ is taken with respect to $V^{\Phi}_{K_{0}}$ and therefore $F^{1}\Phi_{*}(L)$ is considered as an isotropic line in $V_{K_{0}}$ via the following inclusions 
\begin{equation}\label{period}
F^{1}\Phi_{*}(L)\subset \Phi_{*}(L_{x_{0}})\subset  L_{x_{0}}+ \Phi_{*}(L_{x_{0}}) + \cdots +\Phi^{d}_{*}(L_{x_{0}}) \subset V^{\Phi}_{K_{0}}\otimes K_{0}.
\end{equation}
where $d$ is the minimal number that $ L_{x_{0}}+ \Phi_{*}(L_{x_{0}}) + \cdots +\Phi^{d}_{*}(L_{x_{0}})$ is $\Phi$-fixed. 

Let $Def^{\rig}_{\ssp}$ be the rigid analytic space associated to $Def_{\ssp}$, then our next result describes the image of it under the period morphism.

\begin{theorem}\label{thmX}
The image of $Def^{\rig}_{\ssp}(K)$ in  $\mathcal{Q}^{a}(K)$ can be described by the homogeneous space $X=\SO(V^{\Phi}_{K_{0}})/ \SO(W)\times \SO(W^{\perp})$.
\end{theorem}

\begin{proof}
Let $x_{0}\in \breve{\mathcal{M}}_{\ssp}$ be a superspecial point and $x=(F^{1}\Phi_{*}(L)\subset \Phi_{*}(L_{x_{0}}))\in Def^{\rig}_{\ssp}(K)$ be a lift of $x_{0}$.

The restriction of the period map $\pi_{\dR}$ to $Def^{\rig}_{\ssp}(K)$ is given by 
$$F^{1}\Phi_{*}(L)\subset F^{1}L+F^{1}\Phi_{*}(L) \subset  L_{x_{0}}+\Phi_{*}(L_{x_{0}})\subset V^{\Phi}_{K_{0}}\otimes K_{0}.$$ 
This follows from the above discussion and \eqref{period}.  Therefore one sees that an isotropic line $F^{1}L \in \mathcal{Q}^{a}(K)$ lies in the image of $Def^{\rig}_{\ssp}(K)$ if it is given by $F^{1}L\subset F^{1}L+\Phi_{*}(F^{1}L)$ with $W(L)=F^{1}L+\Phi_{*}(F^{1}L)$ stabilized by $\Phi$. That is equivalent to giving an oriented plane $W(L)$ in $V^{\Phi}_{K_{0}}$ with its orientation induced by $F^{1}L$. Since $\SO(V^{\Phi}_{K_{0}})$ acts on the set of oriented planes transitively and the stabilizer such an oriented plane $W$ is given by $\SO(W)\times \SO(W^{\perp})$, the result follows.
\end{proof}

\section{Application to Shimura varieties}
\subsection{Local and global quadratic spaces } Let $(V_{l}, Q_{l})$ be a family of local quadratic spaces indexed by the set of places of $\QQ$. Suppose that over $\RR$ the quadratic space $(V_{\RR}, Q_{\RR})$ has signature $(r,s)$. Recall that the Hasse-Witt invariant of $V_{\RR}$ is $\epsilon(V_{\RR})=(-1)^{s(s-1)/2}$. We assume moreover that the Hasse-Witt invariant $\epsilon(V_{l})=1$ for almost all $l$. A necessary and sufficient condition that this family of local quadratic spaces patches to a global quadratic space is given by the following well-known theorem.

\begin{theorem}\label{patch-quadratic}
There is a quadratic space $(V, Q)$ over $\QQ$ whose localization at a place $l$ is given by $(V_{l}, Q_{l})$ if and only if  $\prod_{l}\epsilon(V_{l})=1$.
\end{theorem}

We fix a global quadratic space $(V, Q)$ over $\ZZ_{(p)}$ with the signature of $(V_{\RR}, Q_{\RR})$ given by $(n-2, 2)$ with $n\geq 3$. In particular $\epsilon(V_{\RR})=-1$. We assume at $p$  we have $\epsilon(V_{p})=1$.  Now we consider a different family of quadratic spaces $(V^{\prime}_{l}, Q^{\prime}_{l})$ with $V^{\prime}_{l}=V_{l}$ for $l\neq p, \mathbb{R}$. Then we set $V^{\prime}_{p}= V_{p, K_{0}}^{\Phi}$ as we did in section $2$ and we let $(V^{\prime}_{\RR},Q^{\prime}_{\RR})$ be a quadratic space whose signature is $(n,0)$. Then, by Theorem \ref{patch-quadratic}, we obtain a global quadratic space $(V^{\prime}, Q^{\prime})$ whose localization at any $l$ is $(V^{\prime}_{l}, Q^{\prime}_{l})$. Let $I=\SO(V^{\prime})$ be the corresponding special orthogonal group and similarly let $I^{\Diamond}=\GSpin(V^{\prime})$ be the corresponding spin similitude group. 

The basic locus of the $\GSpin$ type Shimura variety can be uniformized  by the Rapoport-Zink space $\breve{\calM}^{\Diamond}$ considered in the previous sections. 

\begin{theorem}\label{RZ-spin}
There is an isomorphism of formal $W_{0}$-schemes
$$I^{\Diamond}(\QQ)\backslash \breve{\calM}^{\Diamond}\times G^{\Diamond}(\mathbb{A}^{p})/ K^{\Diamond p}\xrightarrow{\sim} \widehat{\mathscr{S}}^{\Diamond}_{K^{\Diamond}, W_{0}}/ \mathscr{S}^{\Diamond}_{ss}$$
where $I^{\Diamond}(\QQ)$ acts on $\breve{\calM}^{\Diamond}$ via the embedding $I^{\Diamond}(\QQ)\hookrightarrow J_{b^{\Diamond}}(\QQ_{p})$. 
\end{theorem}

\begin{proof}
This is proved in \cite[Theorem 7.2.4]{HP17}.
\end{proof}

Let $\mathscr{S}^{\Diamond}_{ssp}(\FF_{p^{2}})$ be the subset of $\mathscr{S}^{\Diamond}_{ss}(\FF_{p^{2}})$ such that the the Kuga-Satake Abelian variety $(A^{\KS}_{\FF_{p^{2}}}, \lambda^{\KS}_{\FF_{p^{2}}})$ is superspecial.  The above isomorphism induces an bijection 
\begin{equation}
\label{ssp-unif}I^{\Diamond}(\QQ)\backslash \breve{\calM}^{\Diamond}_{\ssp}(\FF_{p^{2}})\times G^{\Diamond}(\mathbb{A}^{p})/ K^{\Diamond p}\xrightarrow{\sim} \mathscr{S}^{\Diamond}_{\ssp}(\FF_{p^{2}}).
\end{equation}

Consider the \'{e}tale covering  $\mathscr{S}^{\Diamond}_{K^{\Diamond}}\rightarrow \mathscr{S}_{K}$ and let $\mathscr{S}_{\ss}$ be the image of $\mathscr{S}^{\Diamond}_{\ss}$ under this covering. We define $\mathscr{S}_{\ssp}$ in a completely similar manner. 

The basic locus of the orthogonal type Shimura variety can be uniformized  by the Rapoport-Zink space $\breve{\calM}$ considered in the previous sections.

\begin{corollary}\label{RZ-orth}
There is an isomorphism of formal $W_{0}$-schemes
$$I(\QQ)\backslash \breve{\calM}\times G(\mathbb{A}^{p})/ K^{p}\xrightarrow{\sim} (\widehat{\mathscr{S}}_{K, W_{0}})/\mathscr{S}_{ss}$$
where $I(\QQ)$ acts on $\breve{\calM}$ via the embedding $I(\QQ)\hookrightarrow J_{b}(\QQ_{p})$. 
The above isomorphism immediately induces a bijection 
\begin{equation*}\label{ssp-unif}I(\QQ)\backslash \breve{\calM}_{\ssp}(\FF_{p^{2}})\times G(\mathbb{A}^{p})/ K^{p}\xrightarrow{\sim} \mathscr{S}_{\ssp}(\FF_{p^{2}}).
\end{equation*} 
\end{corollary}

Combining Theorem \ref{thmZ}, we obtain a double coset uniformization of the superspecial locus of the orthogonal type Shimura variety $\mathscr{S}_{\ssp}(\FF_{p^{2}})$. For this, let $U_{0,p}= \SO(\Lambda, \pm)\subset I(\QQ_{p})$ and $U^{p}=K^{p}\subset I(\mathbb{A}^{p})$. 

\begin{theorem}\label{ssp-shim}
There is a double coset uniformization of the superspecial locus of the orthogonal type Shimura variety given by
$$ \mathscr{S}_{\ssp}(\FF_{p^{2}})\xrightarrow{\sim}I(\QQ)\backslash I(\mathbb{A}_{f}) /U_{0,p}U^{p}$$
\end{theorem}

\begin{proof}
We have an isomorphism  $\mathscr{S}_{\ssp}(\FF)\xrightarrow{\sim}I(\QQ)\backslash Z\times G(\mathbb{A}^{p})/ K^{p}$ using \eqref{ssp-unif} and Theorem \ref{thmZ}. Since there is an isomorphism $G(\mathbb{A}^{p})\cong I(\mathbb{A}^{p})$ and $I(\QQ_{p})=\SO({V^{\Phi}_{K_{0}}})$, the result follows.
\end{proof}

Let  $\mathscr{S}^{\ssp}(\mathcal{O}_{K})$ be the subset of $\widehat{\mathscr{S}}_{K, W_{0}}/\mathscr{S}_{\ss}(\mathcal{O}_{K})$ of whose reduction modulo $p$ lands in $\mathscr{S}_{\ssp}(\FF_{p^{2}})$.  Then Theorem \ref{RZ-orth} induces a bijection
\begin{equation}\label{ssp-lift-unif}I(\QQ)\backslash Def_{\ssp}(\mathcal{O}_{K}) \times G(\mathbb{A}^{p})/ K^{p}\xrightarrow{\sim} \mathscr{S}^{\ssp}(\mathcal{O}_{K}) .\end{equation}

\begin{theorem}\label{ssp-shim-lift}
There is a double coset description of $\mathscr{S}^{\ssp}(\mathcal{O}_{K})$ given by
$$\mathscr{S}^{\ssp}(\mathcal{O}_{K})\xrightarrow{\sim}I(\QQ)\backslash I(\mathbb{A}_{f}) /U_{1,p}U^{p}$$ where $U_{1, p}=\SO(W,\iota)\times \SO(M)$ 
\end{theorem}

\begin{proof}
We have an isomorphism  $\mathscr{S}^{ssp}(\mathcal{O}_{K})\xrightarrow{\sim}I(\QQ)\backslash Y\times G(\mathbb{A}^{p})/ K^{p}$ using \eqref{ssp-lift-unif} and Corollary \ref{corY}. Since there is an isomorphism $G(\mathbb{A}^{p})\cong I(\mathbb{A}^{p})$ and $I(\QQ_{p})=\SO(V^{\Phi}_{p, K_{0}})$, the result follows.
\end{proof}

\subsection{Geometric Mass formula}\label{mass-for}
 Finally, we deduce some quantitive results about the size of the double coset in Theorem \ref{ssp-shim}. This is usually refereed to as the Mass formula in the literature. See \cite{Yu11} for a treatment in the PEL type case. Our modest goal of this short subsection is to provide an analogue of the simple Mass formula in the setting of orthogonal type Shimura varieties. 

Let $x\in \mathscr{S}_{\ssp}$ be a superspecial point and let $g_{x}\in I(\QQ)\backslash I(\mathbb{A}_{f}) /U_{0,p}U^{p}$ be the corresponding class and we choose a representative in $I(\mathbb{A}_{f}) $ denoted by the same symbol. Let $U=U_{0,p}U^{p}$ we define the finite group $\Gamma_{x}$ by the following formula
$$g_{x}Ug^{-1}_{x}\cap I(\QQ)= \Gamma_{x}.$$

We define the arithmetic mass $\Mass^{a}_{\mathscr{S}_{\ssp}}$ for the superspecial locus $\mathscr{S}_{\ssp}$ using the formula
$$\Mass^{a}_{\mathscr{S}_{\ssp}}=\sum_{x\in \mathscr{S}_{\ssp}}\frac{1}{|\Gamma_{x} |}$$
where $|\Gamma_{x}|$ is the size of the group $\Gamma_{x}$. 

Next we give a geometric interpretation of the group $ \Gamma_{g_{x}}$. For this, let $\tilde{x}\in \mathscr{S}^{\Diamond}_{\ssp}$ be a lift of $x$ and let $A_{\tilde{x}}$ be the  abelian variety corresponding to $\tilde{x}\in {\mathscr{S}^{\Diamond}_{ssp}}$. It is equipped with a polarization $\lambda_{\tilde{x}}$ and a prime to $p$ level structure $\epsilon^{p}_{\tilde{x}}$. We will write $\underline{A}_{\tilde{x}}=(A_{\tilde{x}}, \lambda_{\tilde{x}}, \epsilon^{p}_{\tilde{x}})$ for this datum. Let  $\Isom(\underline{A}_{\tilde{x}}, \underline{A}_{\tilde{x}})_{\QQ}$ be the self-quasi-isogenies of $A_{\tilde{x}}$ that preserve the additional structures. Recall there are quadratic spaces $\mathbf{V}_{\tilde{x},\QQ}^{p}\subset \End(V^{p}(A_{\tilde{x}}))$ and $\mathbf{V}_{\cris,\tilde{x}, \QQ}\subset \End(\mathbf{H}_{\cris,\tilde{x}, \QQ})$ defined in \eqref{special-p}, \eqref{special-l} with $\mathbf{V}_{\tilde{x},\QQ}^{p}=\prod_{l\neq p}\mathbf{V}_{l,\tilde{x},\QQ}$ . Define the group $\Isom_{\mathbf{V}}(\underline{A}_{\tilde{x}}, \underline{A}_{\tilde{x}})$ to be the largest closed subgroup of $\Isom(\underline{A}_{\tilde{x}}, \underline{A}_{\tilde{x}})_{\QQ}$ such that its image in $ \End(V^{p}(A_{\tilde{x}}))$ lands in $\GSpin(\mathbf{V}^{p}_{\tilde{x}, \QQ})$ and its image in $\End(\mathbf{H}_{\cris,\tilde{x}, \QQ})$ lands in $\GSpin(\mathbf{V}_{\cris,\tilde{x}, \QQ})$. The group  $\Isom_{\mathbf{V}}(\underline{A}_{\tilde{x}}, \underline{A}_{\tilde{x}})$ is a subgroup of $\GSpin(V^{\prime})$ by \cite[Theorem 6.4]{MP15}. We define the group $\Isom_{\mathbf{V}}(x)$ to be the image of $\Isom_{\mathbf{V}}(\underline{A}_{\tilde{x}}, \underline{A}_{\tilde{x}})$ in $\SO(V^{\prime})$. Using this we define the geometric mass  of $\mathscr{S}_{\ssp}$ to be $$\Mass^{g}_{\mathscr{S}_{\ssp}}=\sum_{x\in \mathscr{S}_{\ssp}}\frac{1}{|\Isom_{\mathbf{V}}(x)|}.$$

\begin{proposition}
The group $\Gamma_{x}$ is  isomorphic to $\Isom_{\mathbf{V}}(x)$ and therefore $$\Mass^{g}_{\mathscr{S}_{\ssp}}=\Mass^{a}_{\mathscr{S}_{\ssp}}.$$
\end{proposition}

\begin{proof}
Notice that $I^{\Diamond}(\QQ)$ is the subgroup of $\Isom(\underline{A}_{\tilde{x}}, \underline{A}_{\tilde{x}})_{\QQ}$ in the isogeny category of $(A_{\tilde{x}},\lambda_{\tilde{x}})$ that preserve the Hodge cycles on $A_{\tilde{x}}$.  Then by definition its image in $ \End(V^{p}(A_{\tilde{x}}))$ lands in $\GSpin(V(\mathbb{A}^{p}_{f}))$ and its image in $\End(\mathbf{H}_{\cris, \tilde{x}, \QQ})$ is $\GSpin(V^{\prime}_{p})$. Since the abelian variety $A_{\tilde{x}}$ is supersingular, it follows from \cite[Theorem 6.4]{MP15} that the image of $\Isom_{\mathbf{V}}(\underline{A}_{\tilde{x}}, \underline{A}_{\tilde{x}})$ in $\GSpin(\mathbf{V}_{\cris, \tilde{x}, \QQ})$ agrees with  $\GSpin(V^{\prime}_{p})$. It is also clear that the image of $\Isom_{\mathbf{V}}(\underline{A}_{\tilde{x}}, \underline{A}_{\tilde{x}})$ in $\GSpin(\mathbf{V}_{l, \tilde{x}, \QQ})$ agrees that of $I^{\Diamond}(\QQ)$ in $\GSpin(\mathbf{V}_{l, \tilde{x}, \QQ}))$ for all $l\neq p$. Then we can conclude that $\Gamma_{x}$ agrees with $\Isom_{\mathbf{V}}(x)$ essentially by Hasse priniciple. 
\end{proof}

\begin{corollary}
Let $\Vol(U)=[I(\widehat{\ZZ}): U]$.
$$\Mass^{g}_{\mathscr{S}_{\ssp}} = \begin{cases}
\Vol(U)\prod^{m}_{r=1}\zeta(1-2r)\frac{1}{2^{m-1}}\frac{p^{2m}-1}{2(p+1)} &\text{if  $n=2m+1$};\\
\Vol(U)\prod^{m}_{r=1}\zeta(1-2r)L(1-m,\chi)\frac{1}{2^{m-1}}\frac{(p^{m-1}+1)(p^{m+1})}{2(p+1)}  &\text{if $n=2m$}.
\end{cases}$$
\end{corollary}

\begin{proof}
This quantity can be calculated explicitly using the following Shimura's Mass formula see \cite{GHY01}. We give some details of this formula below. Notice that $$\Mass^{a}_{\mathscr{S}_{\ssp}}=\sum_{x\in \mathscr{S}_{\ssp}}\frac{1}{|\Gamma_{x} |}=\Vol(U)\Mass(\Lambda)$$ for a suitable maximal lattice $\Lambda$ in $V^{\prime}$. Then the formula in this corollary follows from \cite[Proposition 7.4, Proposition 7.5]{GHY01}. Note that here we applying the formula in \cite[Proposition 7.4, Proposition 7.5]{GHY01} in the case when the quadratic space has unit determinant and Hasse-Witt invariant $-1$. 
\end{proof}

\end{document}